\documentclass[a4paper, 12pt]{amsart}
 \usepackage{amsmath, amsthm, amssymb, amsfonts}
\usepackage[utf8]{inputenc}
\usepackage[normalem]{ulem}
\usepackage[colorlinks = true,
            linkcolor = black,
            urlcolor  = blue,
            citecolor = red,
            anchorcolor = red]{hyperref}
\usepackage{url}
\usepackage{longtable} 
\usepackage{enumitem}
\usepackage{relsize}
\usepackage{bm}
\usepackage{xcolor}
\usepackage{graphicx}
\usepackage{mathdots}
\usepackage{color}
\usepackage{tikz, tikz-cd}
\usetikzlibrary{arrows, decorations.markings}
\usetikzlibrary{positioning}
\tikzstyle{vecArrow} = [thick, decoration={markings,mark=at position
   1 with {\arrow[semithick]{open triangle 60}}},
   double distance=1.4pt, shorten >= 5.5pt,
   preaction = {decorate},
   postaction = {draw,line width=1.4pt, white,shorten >= 4.5pt}]
\tikzstyle{innerWhite} = [semithick, white,line width=1.4pt, shorten >= 4.5pt]
\tikzstyle{vecEq} = [thick,
   double distance=1.4pt,
   preaction = {decorate},
   postaction = {draw,line width=1.4pt, white,shorten >= 4.5pt}]
\usetikzlibrary{matrix}
\usetikzlibrary{shapes}
\usetikzlibrary{arrows,decorations.markings, tikzmark}
\usepackage{stmaryrd}

\pgfarrowsdeclare{bad to}{bad to}
{
  \pgfarrowsleftextend{-2\pgflinewidth}
  \pgfarrowsrightextend{\pgflinewidth}
}
{
  \pgfsetlinewidth{0.8\pgflinewidth}
  \pgfsetdash{}{0pt}
  \pgfsetroundcap
  \pgfsetroundjoin
  \pgfpathmoveto{\pgfpoint{-3\pgflinewidth}{4\pgflinewidth}}
  \pgfpathcurveto
  {\pgfpoint{-2.75\pgflinewidth}{2.5\pgflinewidth}}
  {\pgfpoint{0pt}{0.25\pgflinewidth}}
  {\pgfpoint{0.75\pgflinewidth}{0pt}}
  \pgfpathcurveto
  {\pgfpoint{0pt}{-0.25\pgflinewidth}}
  {\pgfpoint{-2.75\pgflinewidth}{-2.5\pgflinewidth}}
  {\pgfpoint{-3\pgflinewidth}{-4\pgflinewidth}}
  \pgfusepathqstroke
}

\def\Mbar{\overline{M}}

\def\p{\textsf{p}}

\theoremstyle{definition}
\newtheorem{definition}{Definition}
\newtheorem{theorem}[definition]{Theorem}
\newtheorem{example}[definition]{Example}
\newtheorem{proposition}[definition]{Proposition}

\newtheorem{corollary}[definition]{Corollary}

\newtheorem{lemma}[definition]{Lemma}
\newtheorem{conjecture}[definition]{Conjecture}

\newtheorem{remark}[definition]{Remark}

\newcommand{\QMod}{\mathsf{QMod}}
\newcommand{\Mod}{\mathsf{Mod}}

\newcommand{\BA}{{\mathbb{A}}}

\newcommand{\BC}{{\mathbb{C}}}

\newcommand{\BE}{{\mathbb{E}}}

\newcommand{\BH}{{\mathbb{H}}}

\newcommand{\BL}{{\mathbb{L}}}

\newcommand{\BN}{{\mathbb{N}}}

\newcommand{\BP}{{\mathbb{P}}}
\newcommand{\BQ}{{\mathbb{Q}}}
\newcommand{\BR}{{\mathbb{R}}}

\newcommand{\BZ}{{\mathbb{Z}}}

\newcommand{\CO}{{\mathcal O}}
\newcommand{\CP}{{\mathcal P}}

\newcommand{\CS}{{\mathcal S}}

\DeclareMathOperator{\Aut}{Aut}

\newcommand{\red}{\rm red}
\newcommand{\rel}{\rm rel}

\newcommand{\Ee}{\rm E}

\newcommand{\mg}{\overline{M}_{g}}
\newcommand{\mgn}{\overline{M}_{g,n}}
\newcommand{\bF}{{\mathsf{F}}}
\newcommand{\bB}{{\mathsf{B}}}
\newcommand{\bT}{{\mathsf{T}}}
\newcommand{\bU}{{\mathsf{U}}}
\newcommand{\bP}{\mathsf{P}}
\newcommand{\bG}{\mathsf{G}}
\newcommand{\bH}{\mathsf{H}}

\newcommand{\SL}{{\rm SL}}
\newcommand{\GL}{{\rm GL}}

\newcommand{\ev}{\mathop{\rm ev}\nolimits}

\newcommand{\Dq}{{\mathsf{D}}_q}

\newcommand{\one}{\mathsf{1}}
\newcommand{\pt}{{\mathsf{p}}}

\newcommand{\congpf}{\xymatrix@1@=15pt{\ar[r]^-\sim&}}

\newcommand{\pic}{\mathfrak{Pic}_{g,n}}

\newcommand{\dr}{\mathsf{DR}}

\newcommand{\lan}{\big\langle}
\newcommand{\ran}{\big\rangle}
\newcommand{\dc}{\frac{d}{dC_2}}
\newcommand{\degm}{\underline{\deg}}
\newcommand{\identity}{\mathrm{Id}}
\newcommand{\mi}{\phantom{-}}
\newcommand{\unmu}{\underline{\mu}}
\newcommand{\TCS}{\widetilde{\mathcal{S}}}

\usepackage{tikz}
\usepackage{tikz-cd}
\tikzcdset{row sep/normal=0.2cm}
\AtBeginDocument{%
	\def\MR#1{}
}

\begin{document}

\baselineskip=16pt
\parskip=5pt

\title{Curves on K3 surfaces in divisibility two}
\author{Younghan Bae}
\address {ETH Z\"urich, Department of Mathematics}
\email{younghan.bae@math.ethz.ch}
\author{Tim-Henrik Buelles}
\address {ETH Z\"urich, Department of Mathematics}
\email{buelles@math.ethz.ch}
\date{\today}
\maketitle
\begin{abstract} We prove a conjecture of Maulik, Pandharipande, and Thomas expressing the Gromov--Witten invariants of K3 surfaces for divisibility two curve classes in all genus in terms of weakly holomorphic quasimodular forms of level two. Then, we establish the holomorphic anomaly equation in divisibility two in all genus. Our approach involves a refined boundary induction, relying on the top tautological group of the moduli space of smooth curves, together with a degeneration formula for the reduced virtual fundamental class with imprimitive curve classes. We use the double ramification relations with target variety as a new tool to prove the initial condition. The relationship between the holomorphic anomaly equation for higher divisibility and the conjectural multiple cover formula of Oberdieck and Pandharipande is discussed in detail and illustrated with several examples.
\end{abstract}

\setcounter{tocdepth}{1} 
\tableofcontents

\setcounter{section}{-1}

\section{Introduction}\label{sec:introduction} Let $S$ be a complex nonsingular projective K3 surface and $\beta\in H_2(S,\BZ)$ an effective curve class. Gromov--Witten invariants of $S$ are defined via intersection theory on the moduli space $\mgn(S,\beta)$ of stable maps from $n$-pointed genus $g$ curves to $S$. This moduli space comes with a virtual fundamental class. However, the virtual class vanishes for $\beta\neq 0$ so, instead, we use the \textit{reduced class}\footnote{We will identify this class with its image under the cycle class map $A_*\to H_{2*}$.}
\[[\mgn(S,\beta)]^{red}\in A_{g+n}\big(\mgn(S,\beta),\BQ\big)\,.\]
For integers $a_i\geq 0$ and cohomology classes $\gamma_i\in H^*(S,\BQ)$ we define
\[ \big\langle \tau_{a_1}(\gamma_1)\ldots\tau_{a_n}(\gamma_n)\big\rangle_{g,\beta}^S = \int_{[\mgn(S,\beta)]^{red}} \prod_{i=1}^n \psi_i^{a_i}\cup \ev_i^*(\gamma_i)\,, \]
where $\ev_i\colon \mgn(S,\beta) \to S$ is the evaluation at $i$-th marking and $\psi_i$ is the cotangent class at the $i$-th marking. By the deformation invariance of the reduced class, the invariant only depends on the norm $\langle \beta,\beta\rangle$ and the divisibility of the curve class $\beta$.

\subsection{Quasimodularity} Gromov--Witten invariants of K3 surfaces for primitive curve classes are well-understood since the seminal paper by Maulik, Pandharipande, and Thomas~\cite{MPT10}. The invariants are coefficients of weakly holomorphic\footnote{Weakly holomorphic means holomorphic on the upper half plane with possible pole at the cusp $i\infty$.} quasimodular forms with pole of order at most one~\cite[Theorem 4]{MPT10}. For imprimitive curve classes, the quasimodularity is conjectured with the level structure~\cite[Section 7.5]{MPT10}.

The quasimodularity can be stated in a precise sense via elliptic K3 surfaces. Let \[\pi\colon S\to \BP^1\] be an elliptic K3 surface with a section and denote by $B, F\in H_2(S,\BZ)$ the class of the section resp.\ a fiber. For any $m\geq 1$ one defines the descendent potential 
\[ \bF_{g,m}\big(\tau_{a_1}(\gamma_1)\ldots\tau_{a_n}(\gamma_n)\big) = \sum_{h\geq 0} \big\langle \tau_{a_1}(\gamma_1)\ldots\tau_{a_n}(\gamma_n)\big\rangle_{g,mB+hF}^S \, q^{h-m}\,. \]
Note that this generating series involves curve classes $mB+hF$ of different divisibilities, bounded by $m$.

It is convenient to use the following homogenized insertions which will lead to quasimodular forms of pure weight. Let $\one\in H^0(S)$ and $\pt\in H^4(S)$ be the identity resp.\ the point class. Denote \[ W=B+F\in H^2(S)\] and let \[U=\BQ\langle F, W\rangle \subset H^2(S) \] be the hyperbolic plane in $H^2(S)$ and let $U^\perp\subset H^2(S)$ be its orthogonal complement with respect to the intersection form. We only consider second cohomology classes which are pure with respect to the decomposition \[H^2(S,\BQ)\cong \BQ\lan F\ran\oplus \BQ\lan W \ran\oplus U^\perp\,.\] Following \cite[Section 4.6]{BOPY18}, define a modified degree function $\degm$ by
\begin{displaymath}\label{def:degm}
\degm(\gamma) = \left\{ \begin{array}{ll}
 2 & \textrm{if $\gamma=W$ or $\pt$}\,,\\
 1 & \textrm{if $\gamma \in U^{\perp}$}\,,\\
 0 & \textrm{if $\gamma = F$ or $\one$} \,.
  \end{array} \right.
\end{displaymath}

For $m \geq 1$, consider the {\em Hecke congruence subgroup of level $m$}
\[\Gamma_0(m) = \left\{ \begin{pmatrix} a & b \\ c & d \end{pmatrix} \in \SL_2(\BZ)\, \mid c \equiv 0 \mod m \right\}\]
and let $\QMod(m)$ be the space of  quasimodular forms for the congruence subgroup $\Gamma_0(m)\subset \SL_2(\BZ)$. Let $\Delta(q)$ be the modular discriminant
\[ \Delta(q) = q \prod_{n\geq 1} (1-q^n)^{24} \,.\]
Our first main result proves level two quasimodularity of $\bF_{g,2}$, previously conjectured by Maulik, Pandharipande, and Thomas~\cite[Section 7.5]{MPT10}.
\begin{theorem}\label{thm:qmod}
Let $\gamma_1,\ldots, \gamma_n \in H^*(S)$ be homogeneous on the modified degree function $\degm$. Then $\bF_{g,2}$ is the Fourier expansion of a quasimodular form
\[\bF_{g,2}\big(\tau_{a_1}(\gamma_1)\ldots\tau_{a_n}(\gamma_n)\big)\in \frac{1}{\Delta(q)^2} \QMod(2) \]
of weight $2g-12 + \sum_i\degm (\gamma_i)$ with pole at $q=0$ of order at most $2$.
\end{theorem}

\subsection{Holomorphic anomaly equation}
In the physics literature, the (conjectural) {\em holomorphic anomaly equation}~\cite{BCOV93,BCOV94} predicts hidden structures of the Gromov--Witten partition function associated to Calabi--Yau varieties. For the past few years, there has been an extensive work to prove the holomorphic anomaly equation in many cases: local $\BP^2$~\cite{LP18}, the quintic threefold \cite{CGL18,GJR18}, K3 surface with primitive curve classes~\cite{OP18}, elliptic fibration~\cite{OP19} and $\BP^2$ relative to a smooth cubic~\cite{BFGW20}. 

Every quasimodular form for $\Gamma_0(m)$ can be written uniquely as a polynomial in $C_2$ with coefficients which are modular forms for $\Gamma_0(m)$ \cite[Proposition 1]{KZ95}. Here, \[C_2(q)= -\frac{1}{24}E_2(q)\] is the renormalized second Eisenstein series. Assuming quasimodularity, the holomorphic anomaly equation fixes the non-holomorphic parameter of the Gromov--Witten partition function of K3 surfaces in terms of lower weight partition functions: it computes the derivative of $\bF_{g,m}$ with respect to the $C_2$ variable. See \cite{OP18} for the proof of holomorphic anomaly equation for K3 surfaces with primitive curve classes and \cite{OP19} for the holomorphic anomaly equation associated to elliptic fibrations.

Define an endomorphism~\cite[Section 0.6]{OP18} \[\sigma\colon H^*(S^2) \to H^*(S^2)\]
by the following assignments: \[\sigma(\gamma \boxtimes \gamma')=0\]
if $\gamma$ or $\gamma'\in H^0(S)\oplus \BQ\lan F \ran \oplus H^4(S)$, and for $\alpha,\alpha'\in U^{\perp}$, 
\begin{align*}
    &\sigma(W\boxtimes W)=\Delta_{U^{\perp}}, \; \sigma(W\boxtimes \alpha)= -\alpha\boxtimes F,\\
    & \sigma(\alpha\boxtimes W)=-F\boxtimes\alpha,\; \sigma(\alpha,\alpha')=\langle\alpha,\alpha'\rangle F\boxtimes F\,,
\end{align*}
where $\Delta_{U^\perp}$ denotes the diagonal class for the intersection pairing on~$U^\perp$. We will view $\sigma$ as the exterior product $\sigma_1\boxtimes\sigma_2$ via K\"unneth decomposition.

Recall the virtual fundamental class for trivial curve classes which will play a role for the holomorphic anomaly equation. For $\beta=0$ we have an isomorphism
\[\mgn(S,0)\cong \mgn\times S\]
and the virtual class is given by
\begin{equation*}
    [\mgn(S,0)]^{vir}=\begin{cases}
               [\Mbar_{0,n}\times S] & \textup{if $g=0$}\,,\\
               c_2(S)\cap [\Mbar_{1,n}\times S] & \textup{if $g=1$}\,,\\
               0 & \textup{if $g\geq2$}\,.
            \end{cases}
\end{equation*}
Also, consider the pullback under the morphism $\pi\colon S \to \BP^1$ of the diagonal class of $\BP^1$
\[\Delta_{\BP^1} = 1\boxtimes F + F \boxtimes 1 = \sum_{i=1}^2 \delta_i\boxtimes \delta_i^\vee\,.\]
Define the generating series\footnote{Here, instead of descendent insertions we use a tautological class $\alpha\in R^*(\mgn)$, see the comment in Section~\ref{sec:compatibilityI}}
\begin{align}
    &\bH_{g,m}\big(\alpha;\gamma_1,\ldots,\gamma_n\big) \label{eq:fullHAE}\\ 
    &= \bF_{g-1,m}\big(\alpha;\gamma_1,\ldots,\gamma_n,\Delta_{\BP^1}\big)\nonumber \\
    &\mi+2\sum_{\substack{g=g_1+g_2\\ \{1,\ldots,n\}=I_1\sqcup I_2\\ \,i\in\{1,2\}}}\bF_{g_1,m}\big(\alpha_{I_1};\gamma_{I_1},\delta_i\big)\,\bF^{vir}_{g_2}\big(\alpha_{I_2};\gamma_{I_2},\delta_i^{\vee}\big) \nonumber\\
    &\mi-2\sum_{i=1}^n \bF_{g,m}\big(\alpha\psi_i;\gamma_1,\ldots,\gamma_{i-1},\pi^*\pi_*\gamma_i,\gamma_{i+1},\ldots,\gamma_n\big)\nonumber \\
    &\mi+\frac{20}{m}\sum_{i=1}^n\langle\gamma_i,F\rangle\bF_{g,m}\big(\alpha;\gamma_1,\ldots,\gamma_{i-1},F,\gamma_{i+1},\ldots,\gamma_n\big)\nonumber\\
    &\mi-\frac{2}{m}\sum_{i<j}\bF_{g,m}\big(\alpha;\gamma_1,\ldots,\underbrace{\sigma_1(\gamma_i,\gamma_j)}_{i \textup{th}},\ldots,\underbrace{\sigma_2(\gamma_i,\gamma_j)}_{j \textup{th}},\ldots,\gamma_n\big)\,,\nonumber
\end{align}
where $\bF^{vir}$ denotes the generating series for virtual fundamental class. In most cases this term vanishes. The equation takes almost the same form for arbitrary $m$, only the last two terms acquire a factor of $\frac{1}{m}$. The appearance of these factors is explained in Section~\ref{sec:HAEm}, see also Example~\ref{ex:HAEm}. We conjecture that the holomorphic anomaly equation has the following form:
\begin{conjecture}\label{conj:HAEm}
\begin{equation}\label{eq:hae}
    \dc \bF_{g,m}\big(\alpha;\gamma_1,\ldots,\gamma_n\big) = \bH_{g,m}\big(\alpha;\gamma_1,\ldots,\gamma_n\big) \,.
\end{equation}
\end{conjecture}
For primitive curve classes, the holomorphic anomaly equation is proven in \cite{OP18}. In higher divisiblity, it is precisely equation~\eqref{eq:hae} that would be implied by the conjectural multiple cover formula for imprimitve Gromow--Witten invariants of K3 surfaces. We explain this in the following section. We prove Conjecture~\ref{conj:HAEm} unconditionally when $m=2$:
\begin{theorem}\label{thm:hae}
For any $g\geq 0$,
\begin{equation}\label{eq:hae2}
    \dc \bF_{g,2}\big(\alpha;\gamma_1,\ldots,\gamma_n\big) = \bH_{g,2}\big(\alpha;\gamma_1,\ldots,\gamma_n\big) \,.
\end{equation}
\end{theorem}
\subsection{Multiple cover formula}
Motivated by the Katz--Klemm--Vafa (KKV) formula, Oberdieck and Pandharipande conjectured a formula which computes imprimitive invariants from the primitive invariants:
\begin{conjecture}(\cite[Conjecture C2]{OP16})\label{conj:georgrahul}
For a primitive curve class $\beta$,
\begin{align} &\big\langle \tau_{a_1}(\gamma_1)\ldots\tau_{a_n}(\gamma_n)\big\rangle_{g,m\beta} \label{eq:mcf}\\
&= \sum_{d\mid m} d^{2g-3+\deg} \big\langle \tau_{a_1}(\varphi_{d,m}(\gamma_1))\ldots\tau_{a_n}(\varphi_{d,m}(\gamma_n))\big\rangle_{g,\varphi_{d,m}\left(\frac{m}{d}\beta\right)}\,.\nonumber \end{align}
\end{conjecture}\medskip 

The invariants on the right hand side are with respect to primitive curve classes\footnote{Section~\ref{sec:mcf} contains all relevant definitions.}. Assuming this formula, we can deduce the holomorphic anomaly equation:

\begin{proposition}\label{prop:mcfimplieshae}
Let $m\geq 1$. Assume the multiple cover formula~\eqref{eq:mcf} holds for all curve classes of divisibility $d\mid m$ and all descendent insertions. Then the holomorphic anomaly equation~\eqref{eq:hae} holds. \end{proposition}

Given this proposition, it seems a natural strategy to prove the multiple cover formula in divisibility two and deduce, as a consequence, the holomorphic anomaly equation. Indeed, our method does follow this logic for $m=2$ and for low genus: we verify the multiple cover formula for $g\leq 2$, see Example \ref{ex:genus2}. For higher genus, however, our method does not seem suitable to achieve this. Instead, our proof of Theorem~\ref{thm:qmod} provides an algorithm, based on the degeneration to the normal cone of a smooth elliptic fiber $E\subset S$, to reduce divisibility two invariants to low genus invariants for which the multiple cover formula is known\footnote{The genus $0$ and genus $1$ cases are proved by Lee and Leung in~\cite{LL05,LL06}. Their proof involves a degeneration formula in symplectic geometry which is not possible in algebraic geometry. We present an algebro-geometric approach using the KKV formula.}. The degeneration formula intertwines invariants of $S$ with invariants of $\BP^1\times E$ in a non-trivial way. This phenomenon is illustrated in Example~\ref{ex:genus2} for the genus~$2$ invariants
\[\big\langle \tau_0(\pt)^2 \big\rangle_{2,2\beta}\,.\]

\subsection{Hecke operator}\label{sec:hecke}
In Section~\ref{sec:mcf} we apply Conjecture~\ref{conj:georgrahul} to an elliptic K3 surface to deduce a conjectural multiple cover formula for the descendent potentials $\bF_{g,m}$. The multiple cover formula for any divisibility~$m$ is then simply a \emph{Hecke operator of the wrong weight} acting on the primitive potential~$\bF_{g,1}$. Indeed, the weight of $\bF_{g,1}$ (and conjecturally of $\bF_{g,m}$) is $2g-12+\degm$, whereas the Hecke operator has the weight of a descendent potential attached to elliptic curves, namely $2g-2+\degm$. This operator can be expressed in terms of Hecke operators (of the correct weight) and translation $q\mapsto q^d$. Together with the holomorphic anomaly equation for primitive curve classes~\cite{OP18} this naturally leads to the above conjecture for the holomorphic anomaly equation for higher divisibility.

\subsection{Plan of the paper}  We prove the quasimodularity and the holomorphic anomaly equation by induction on the genus and the number of markings. In Section~\ref{sec:operators}, we discuss Hecke theory for weakly holomorphic quasimodular forms. This leads to a natural formulation of the multiple cover formula in Section~\ref{sec:mcf} and the imprimitive holomorphic anomaly equation in Section~\ref{sec:HAEm}. In Section~\ref{sec:RelativeHAE}, compatibility of the holomorphic anomaly equation with the degeneration formula is presented. In Section~\ref{sec:InitialCondition}, we derive the multiple cover formula, which implies the holomorphic anomaly equation, for genus~$0$, genus~$1$ and some genus~$2$ decendent invariants from the KKV formula. The genus~$2$ computation relies on double ramification relations with target variety. This result serves as the initial condition for our induction.
In Section~\ref{sec:ProofQmodHAE}, we use previous results to prove Theorem~\ref{thm:qmod} and \ref{thm:hae}. The property of the top tautological group  $R^{g-1}(M_{g,n})$ reduces higher genus cases to lower genus invariants discussed in Section~\ref{sec:InitialCondition}. 


\subsection*{Acknowledgements}
We are grateful to G.\ Oberdieck, R.\ Pandharipande, J.\ Shen, L.\ Wu and Q.\ Yin for many discussions on the Gromov--Witten theory of K3 surfaces. We want to thank D.\  Radchenko for useful suggestions on quasimodular forms. Finally, we thank the anonymous referee whose suggestion greatly improved this article.

Y.~B. was supported by ERC-2017-AdG-786580-MACI and the Korea Foundation for Advanced Studies. T.-H.~B. was supported by ERC-2017-AdG-786580-MACI.

The project has received funding from the European Research Council (ERC) under the European Union Horizon 2020 research and innovation program (grant agreement No. 786580).

\section{Quasimodular forms and Hecke operators}\label{sec:operators}
We recall basic properties of quasimodular forms and Hecke operators, see~\cite{Ko93,Z08}, in particular~\cite[pp.\ 156--163]{Ko93} and \cite[Ch.\ 3, Section~3]{Ko93}. The Hecke theory for weakly holomorphic quasimodular forms however seems to be less well documented. We thus also include some proofs.

The following operators will play a central role. For any Laurent series
\begin{equation}\label{eq:Laurent}f(q) = \sum_{n=-\infty}^{\infty} a_n q^n \end{equation}
and $d\in\BZ_{>0}$ we define
\[ \Dq f = q \frac{d}{dq} f\,, \quad \bB_d f= \sum_{n=-\infty}^{\infty}a_{n} q^{dn}\, ,\quad \bU_d f=\sum_{n=-\infty}^{\infty} a_{dn} q^n\, . \]

We will apply these operators to the Laurent series associated to certain modular functions. For this we briefly review the definition of modular forms.

\subsection{Quasimodular forms}
Let $\BH=\{\tau\in \BC \mid \rm{Im}(\tau)>0\}$ be the upper half-plane. The group $\GL^+_2(\BR)$ of real $2\times 2$-matrices with positive determinant acts on $\BH$ via 
\[ A \tau = \frac{a\tau+b}{c\tau +d}\,,\quad A=\begin{pmatrix} a&b\\c&d\end{pmatrix}\in \GL^+_2(\BR)\,. \] 
Let $f\colon \BH \to \BC$ be a function and let
\[ q=e^{2\pi i \tau}\,, \quad y=\rm{Im}(\tau)\,.\]
For $k\in\BZ$ define the \emph{$k$-th slash operator} 
\[ (f|_k A) (\tau) = \det(A)^{k/2}(c\tau+d)^{-k} f(A\tau)\,. \]
 
\begin{definition} \label{def:qmod}
A {\em quasimodular form} of weight~$k$ for $\SL_2(\BZ)$ is a holomorphic function~$f\colon \BH \to \BC$ admitting a Fourier expansion
\begin{equation}\label{eq:Fourier} f(q) = \sum_{n=0}^{\infty} a_n q^n\,,\quad |q|<1\,,\end{equation}
such that there exist $p\geq 0$ and holomorphic functions~$f_r$, $r=0,\ldots,p$ satisfying the following conditions:
\begin{enumerate}[label=(\roman*)]
\item the (non-holomorphic) function $\widehat {f}= \sum_{r=0}^{p} f_r y^{-r}$
satisfies the transformation law
\[ \widehat{f}|_k\gamma = \widehat{f} \textrm{ for all } \gamma\in\SL_2(\BZ)\,,\]
\item $f=f_0$,
\item each $f_r$ has an expansion of the form~\eqref{eq:Fourier}.
\end{enumerate}

If $p=0$ then $f$ is called a {\em modular form}. We denote the space of modular resp.\ quasimodular forms by $\Mod$ and $\QMod$.

\end{definition}
\begin{remark}\label{rem:qmod}
 If $\widehat{f}= \sum_{r=0}^{p} f_r y^{-r}$ as above with $f_p\neq 0$, then each $f_r$ is a quasimodular form of weight $k-2r$, see~\cite[Proposition 20]{Z08}. Moreover, the last one, i.e.\ $f_p$ is in fact modular (of weight~$k-2p$). The following structural results are well-known
~\cite[Proposition 4, Proposition 20]{Z08}
 \[ \Mod = \BC[C_4,C_6]\,, \quad \QMod=\BC[C_2,C_4,C_6]\,,\]
 where \[C_{2i}(q) = -\frac{B_{2i}}{2i\cdot(2i)!}E_{2i}(q)\] is the renormalized $2i$-th Eisenstein series. The notion (i) defines the space $\mathsf{AHM}$ of \emph{almost holomorphic modular forms} and the assignment $\widehat{f}\mapsto f$ is an isomorphism
 \[ \mathsf{AHM} \to \QMod\,.\]
 Under this map, differentiation with respect to $\frac{1}{8\pi y}$ corresponds to differentiation with respect to $C_2$.
\end{remark}
 
The modular functions considered in this paper will usually have poles at the cusp $\tau=i\infty$ corresponding to $q=0$. We will refer to these functions as \emph{weakly holomorphic} with pole of specified order. We want to clarify this terminology in the context of quasimodular forms.

\begin{definition}
A function $f$ is said to be {\em weakly holomorphic quasimodular with pole of order at most}~$m\geq 0$, if $f$ satisfies the conditions in Definition~\ref{def:qmod} except that each $f_r$ is allowed to have a pole at the cusp $i\infty$ of order at most~$m$. If $p=0$ then $f$ is called a weakly holomorphic modular form with pole of order at most~$m$.
\end{definition}

By parallel arguments as in~\cite[Proposition 20]{Z08}, the assertions in Remark~\ref{rem:qmod} hold analogously for weakly holomorphic quasimodular forms. In particular, $f_p$ is weakly holomorphic modular with pole of order at most~$m$. The space of weakly holomorphic modular forms is generated  by $\frac{1}{\Delta}$ over $\Mod$, where
\[ \Delta(q) = q \prod_{n\geq 1} (1-q^n)^{24}\]
is the modular discriminant.\footnote{See~\cite{DJ08} where the authors examine an explicit basis of the space of weakly holomorphic modular forms.} As a consequence, 
\[f_p\in \frac{1}{\Delta^m}\Mod\]
and since $f_p$ is of weight~$k-2p$ (and there are no non-zero modular forms of negative weight) we have $k\geq 2p-12m$.

For quasimodular forms we include the following observation.

\begin{lemma}\label{lem:weaklyholo}
The space of weakly holomorphic quasimodular forms with pole of order at most~$m$ is given by
\[ \frac{1}{\Delta^m}\QMod\,.\]
\end{lemma}
\begin{proof}

Let $f$ be a weakly holomorphic form with pole of order at most~$m$ and weight~$k$ and let \[ \widehat{f}= \sum_{r=0}^{p} f_r y^{-r}\,, \] 
with $f=f_0$. Multiplying by $\Delta^m$ we have for all $\gamma\in\SL_2(\BZ)$
\[ (\Delta^m \widehat{f})|_{k+12m}\gamma = (\Delta^m)|_{12m}\gamma \cdot (\widehat{f})|_k\gamma = \Delta^m \widehat{f}\,. \]
Since each $\Delta^m f_r$ is holomorphic at $i\infty$ this proves
\[ f \in \frac{1}{\Delta^m}\QMod\,. \]
Analogous argument shows that the quotient of any quasimodular form by $\Delta^m$ defines a weakly holomorphic quasimodular form with pole of order at most~$m$.
\end{proof}
 
\subsection{Hecke operators}
Let $m\in\BN$ and consider the set of integral matrices of determinant $m$
\[ H_m = \Big\{ \begin{pmatrix} a&b\\c&d\end{pmatrix} \mid a,b,c,d\in\BZ\,, ad-bc=m\Big\} \,.\]
The modular group $\SL_2(\BZ)$ acts on $H_m$ by left multiplication. The classical \emph{Hecke operators} $\bT_m$ acting on modular forms~$f$ of weight~$k$ are defined by~\cite[Section 4.1]{Z08}
\[ \bT_mf = m^{k/2-1} \sum_{\gamma\in \SL_2(\BZ)\setminus H_m} f|_k\gamma\,.\]
This definition is equivalent to~\cite[Ch.\ 3, Proposition 38]{Ko93}
\begin{equation}\label{eq:Hecke} \bT_m = \sum_{ad=m} a^{k-1} \bB_a\bU_d\,. \end{equation}
The action of \eqref{eq:Hecke} naturally extends to the action of the $q$-expansion of weakly holomorphic quasimodular forms. We prove that the action again defines a weakly holomorphic quasimodular form. For simplicity (we will only use this case) we restrict to the case when $f$ has a pole of order at most one.

\begin{lemma}
Let $f\in\frac{1}{\Delta}\QMod$ be of weight~$k$. Then $\bT_mf$ is a weakly holomorphic quasimodular form of weight~$k$ with pole of order at most~$m$, i.e.\ 
\[ \bT_m f\in \frac{1}{\Delta^m}\QMod\,. \]
\end{lemma}
\begin{proof}
In~\cite{Mo08} it is shown that $\bT_m$ defines a map $\QMod\to\QMod$ preserving the weight. We briefly recall the key arguments for $f\in\QMod$. The definition of quasimodular forms is equivalent to the condition\footnote{This notion is called `differential modular form' in~\cite{Mo08}. As pointed out in~\cite[Section 5.3]{Z08}, this notion is equivalent to be a quasimodular form.}
\[ (f|_k\gamma) (\tau) = \sum_{r=0}^p \left(\frac{c}{c\tau+d}\right)^r f_r(\tau) \textrm{ for all } \gamma=\begin{pmatrix} a&b\\c&d\end{pmatrix}\in\SL_2(\BZ)\,, \]
where $f_r$ are as in Definition~\ref{def:qmod}. Defining a modification of the slash operator for quasimodular forms\footnote{This definition differs from~\cite[Equation 12]{Mo08} by a factor $m^{-p}$, where~$p$ is the depth of~$f$. Our definition of the Hecke operator differs by the same factor.}
\[ (f||_k A) (\tau) = \sum_{r=0}^p (-c)^r (c\tau+d)^r (f_r|_k A)(\tau)\textrm{ for } A=\begin{pmatrix} a&b\\c&d\end{pmatrix}\in\GL^+_2(\BR)\,, \]
then the quasimodularity is equivalent to
\[ f||_k\gamma = f \textrm{ for all } \gamma\in \SL_2(\BZ)\,.\]
This leads to a parallel treatment of Hecke operators as in the classical context of modular forms. By \cite[Proposition 2]{Mo08} we have
\[ f||_k (\gamma A) = f||_k A\,, \textrm{ for all }\gamma\in\SL_2(\BZ)\,, A\in \GL^+_2(\BR) \]
and we define
\[ \bT_mf = m^{k/2-1} \sum_{A\in \SL_2(\BZ)\setminus H_m} f||_k A\,.\]

This definition is then independent of a choice of representatives of $\SL_2(\BZ)\setminus H_m$. To conclude that $\bT_m f$ is a quasimodular form, we would like to argue that it is invariant under $(-)||_k\gamma$ for all $\gamma\in\SL_2(\BZ)$. This statement, however, is not sensible at the moment\footnote{We are grateful to the referee for pointing out this subtle detail.} because the definition of $(-)||_k\gamma$ relies on the existence of associated functions $f_r$. This technicality is resolved in~\cite[Section 2.4, 2.5]{Mo08} by considering a certain period domain $\CP$ and identifying quasimodular forms as holomorphic functions on $\CP$, which are left $\SL_2(\BZ)$-invariant and satisfy a transformation property for a right action of the subgroup of upper triangular matrices. The domain $\CP$ is contained in $\GL_2(\BC)$ and it contains the upper-half plane $\BH$. The actions are given by left resp.\ right multiplication. The argument carries over to weakly holomorphic quasimodular forms without change.

A particular set of representatives for $\SL_2(\BZ)\setminus H_m$ is given by
\[ \Big\{\gamma_b=\begin{pmatrix} a&b\\0&d\end{pmatrix} \mid a,d\in\BN, ad=m, 0\leq b<d \Big\}\,. \]
Note that $(-)||_k\gamma_b = (-)|_k\gamma_b$ because the terms for $r>0$ vanish. Since
\[ \bU_d f (\tau) = \frac{1}{d}\sum_{0\leq b<d} f\left(\frac{\tau+b}{d}\right)\,, \]
we thus recover equation~(\ref{eq:Hecke}):
\begin{align*}\bT_m f(\tau) &= m^{k/2-1} \sum_{ \substack{ad=m\\0\leq b<d}} d^{-k} m^{k/2} f\left(\frac{a\tau +b}{d}\right)\\
&= \sum_{ad=m} a^{k-1} \bB_a\bU_d f (\tau)\,.
\end{align*}

For weakly holomorphic quasimodular forms $f\in\frac{1}{\Delta}\QMod$ we follow the same proof. The difference here is that the functions $f_r$ are allowed to have simple poles at $i\infty$. The slash operator $(-)||_k$ however may turn a simple pole into a pole of higher order. For $(-)||_k\gamma_b$ this order is bounded by $m$. As a consequence, $\bT_mf$ is weakly holomorphic quasimodular with pole of order at most~$m$.
\end{proof}

For our study of the multiple cover formula in Section~\ref{sec:mcf} we will require a more flexible notion, where the exponent is not necessarily related to the weight. The action of this operator will preserve the weight of weakly holomorphic quasimodular forms, it will, however, introduce poles and level structure.
\begin{definition}\label{def:wrongweighthecke}For $\ell\in\BZ$, we define
\[ \bT_{m,\ell} = \sum_{ad=m} a^{\ell-1} \bB_a\bU_d\, . \]
\end{definition}


The operator $\bT_{m,\ell}$ is simply the $m$-th Hecke operator of weight~$\ell$, which we let act on functions of weight~$k$. By M\"obius inversion we may rewrite each of them in terms of the other (see~\cite[Section 2.7]{Ap76}). For this, let $\mu$ be the M\"obius function.

\begin{lemma}\label{lem:hecke} The action of $\bT_{m,\ell}$ on weakly holomorphic quasimodular forms of weight $k$ is given by \[ \bT_{m,\ell}= \sum_{ad=m} c_{k,\ell}(a) \bB_a \bT_d\, ,\]
 where
 \[ c_{k,\ell}(a) = \sum_{r \mid a} r^{\ell-1} \mu\left(\frac{a}{r}\right)\left(\frac{a}{r}\right)^{k-1}\, . \]
 \end{lemma}
 
 \begin{proof}
 The formula for $c_{k,\ell}$ above can be rewritten as \[c_{k,\ell}= \identity_{\ell-1} \star (\mu\cdot\identity_{k-1})\,,\]
 where $\identity_{\ell-1}(n) = n^{\ell-1}$ is the $(\ell -1)$-th power function and $\star$ denotes Dirichlet convolution, i.e.\ for functions $g,h$ we have
 \[ (g\star h) (m) = \sum_{ad=m}g(a)h(d) \,.\]
 Note also that $\bB$ is multiplicative with respect to composition, i.e.\ for $e\mid a$ we have $\bB_a=\bB_e \bB_{\frac{a}{e}}$ and therefore
  \begin{align*} \bT_{m,\ell} &= \sum_{ad=m} a^{\ell-1} \bB_a\bU_d \\
 &= \sum_{ad=m} \big(\identity_{\ell-1}\star (\mu\cdot \identity_{k-1})\star \identity_{k-1} \big)(a) \bB_a\bU_d \\
 &=\sum_{ad=m} \left( \sum_{e\mid a} c_{k,\ell}(e) \left(\frac{a}{e}\right)^{k-1} \right)\bB_a\bU_d  \\
 &=\sum_{uw=m} c_{k,\ell}(u) \bB_u \left(\sum_{v\mid w} v^{k-1} \bB_{v}\bU_{\frac{w}{v}}\right) \\
 &=\sum_{uw=m} c_{k,\ell}(u) \bB_u\bT_w\,.
 \end{align*}

\end{proof}
 
As a consequence we obtain the following result. Here, we let $\Mod(m)$ and $\QMod(m)$ be the space of modular resp.\ quasimodular forms for the congruence subgroup $\Gamma_0(m)\subset \SL_2(\BZ)$, see the introduction.
 
 \begin{proposition}\label{prop:level}
 Let $f\in\frac{1}{\Delta}\QMod$ be of weight~$k$, then $\bT_{m,\ell}f$ is a weakly holomorphic quasimodular of weight~$k$ with pole of order at most~$m$ for the congruence subgroup $\Gamma_0(m)\subset\SL_2(\BZ)$
 \[\bT_{m,\ell} f\in \frac{1}{\Delta^m}\QMod(m)\,.\]
 \end{proposition}
 
 \begin{proof}
 We use the formula in Lemma~\ref{lem:hecke} and treat each summand separately. By Lemma~\ref{lem:weaklyholo} each $\bT_df$ satisfies
 \[ \bT_d f \in \frac{1}{\Delta^d}\QMod\,. \]
 The action of $\bB_a$ raises $q\mapsto q^a$, or equivalently $\tau\mapsto a\tau$, so it maps $\QMod$ to $\QMod(a)$, see~\cite[Ch.\ 3, Proposition 17]{Ko93}. Therefore
 \[ \bB_a\bT_d f \in \frac{1}{\Delta(q^a)^d} \QMod(a)\,.\]
 Finally, the weakly holomorphic modular form for $\Gamma_0(a)$ defined by
 \[ \frac{\Delta(q)^a}{\Delta(q^a)}\]
 is in fact holomorphic at $i\infty$, i.e.\ contained in $\Mod(a)$. Hence the same is true for its $d$-th power and we find 
 \[ \bB_a \bT_d f \in \frac{1}{\Delta^m}\QMod(a)\,. \]
 which concludes the proof since $\QMod(a)\subset\QMod(m)$.
 \end{proof}
 For later reference, we list the following basic commutator relations between the above operators acting on weakly holomorphic quasimodular forms~$f$ of weight~$k$. Recall, that the algebra $\QMod(m)$ is freely generated by the Eisenstein series $C_2$ over the algebra $\Mod(m)$ of modular forms. Formal differentiation with respect to $C_2$ is therefore well-defined.

\begin{lemma} \label{lem:operators}Let $d$, $e\in \BN$ and $\ell\in\BZ$, then
 	\begin{enumerate}[itemsep=6pt, label=(\roman*)]
 		\item $\bB_d \bB_e = \bB_{de} = \bB_e\bB_d$\,,
 		\item $\bU_d \bU_e = \bU_{de} = \bU_e\bU_d$\,, 
 		\item $\Dq \bB_d = d\, \bB_d \Dq$\,, \quad $\bU_d \Dq = d\, \Dq \bU_d$\,,
 		\item $\bT_{m,\ell+2} \Dq = m\, \Dq \bT_{m,\ell}$\,,
 		\item $\dc \bT_{m,\ell+2} = m \, \bT_{m,\ell} \dc$\,,
 		\item $[\dc,\Dq] = -2k$.
 	\end{enumerate}
 \end{lemma}
\begin{proof}
    The proof for~(i)-(iv) follows directly from the definition. For~(v) one may use that under the isomorphism $\widehat{f}\mapsto f$ the differentiation $\dc$ corresponds to differentiation with respect to $\frac{1}{8\pi y}$, see Remark~\ref{rem:qmod}. The statement~(v) is then checked as an identity of Laurent series in $q$ with polynomial coefficients in $y^{-1}$. The commutator relation~(vi) is well-known, see e.g.\ \cite[Section 5.3]{Z08}.

	
\end{proof}

\section{Multiple cover formula} \label{sec:mcf}
This section contains a discussion of the multiple cover formula. We start by recalling the conjecture formulated in~\cite{OP16}. Then, we study the conjecture for the descendent potentials associated to elliptic K3 surfaces. The result is expressed in terms of Hecke operators. The discussion naturally leads to a candidate for the holomorphic anomaly equation in higher divisibility. We conclude with a proof of the multiple cover formula in fiber direction.
\subsection{Multiple cover formula} Let $S$ be a nonsingular projective K3 surface, $\beta\in H_2(S,\BZ)$ be a \textit{primitive} effective curve class, $m\in \BN$ and $d\mid m$ be a divisor of $m$. The proposed formula by Oberdieck and Pandharipande involves a choice of a real isometry
\[\varphi_{d,m}\colon \Big(H^2(S,\BR)\, ,\langle\, , \,\rangle\Big) 
\to\Big(H^2(S_d, \BR)\, ,\langle\, , \,\rangle \Big) \] between two K3 surfaces such that 
\[ \varphi_{d,m}\left( \frac{m}{d} \beta \right) \in H_2(S_d,\BZ)\]
is a primitive effective curve class\footnote{We view curve classes also as cohomology classes under the natural isomorphism  $H_2(S,\BZ)\cong H^2(S,\BZ)$.}. In \cite{B19} the second author proved that such an isometry can always be found and Gromov--Witten invariants are in fact independent of the choice of isometry.

Consider integers $a_i\in\BN$, cohomology classes $\gamma_i\in H^*(S,\BQ)$ and let $\deg = \sum \deg(\gamma_i)$. Then, the conjectured multiple cover formula~\cite[Conjecture C2]{OP16}, identical to Conjecture~\ref{conj:georgrahul} in Section~\ref{sec:introduction}, is
\begin{align*} &\big\langle \tau_{a_1}(\gamma_1)\ldots\tau_{a_n}(\gamma_n)\big\rangle_{g,m\beta} \\
&= \sum_{d\mid m} d^{2g-3+\deg} \big\langle \tau_{a_1}(\varphi_{d,m}(\gamma_1))\ldots\tau_{a_n}(\varphi_{d,m}(\gamma_n))\big\rangle_{g,\varphi_{d,m}\left(\frac{m}{d}\beta\right)}\,. \end{align*}

Let $S$ be an elliptic K3 surface with a section\footnote{Notations here are as in Section~\ref{sec:introduction}. In particular, we use the modified degree function $\degm$.}. The full (reduced) Gromov--Witten theory of K3 surfaces is captured by $S$ with curve class $mB+hF$ via standard deformation arguments using the Torelli theorem. In fact, the multiple cover conjecture can be captured entirely via $S$ as well: we may choose \emph{the same $S_d=S$} for any $d$ dividing $m$ and $h$. For $l\in\BQ^*$ we define 
\[\phi_{l}\colon H^*(S,\BQ) \to H^*(S,\BQ)\]
acting on $U=\BQ\langle F, W\rangle$  as
\[ \phi_{l}(F) = \frac{1}{l} F\, , \qquad \phi_{l}(W) = l W\,,\]
and trivially on the orthogonal complement $U^{\perp}$. For $d\mid m$ and $d\mid h$ we may choose $\varphi_{d,m}$ as $\phi_{\frac{d}{m}}$:
\[ \phi_{\frac{d}{m}} \left( \frac{m}{d}B + \frac{h}{d}F\right) = B+\left(\frac{m(h-m)}{d^2}+1\right) F \; \textup{ in } H_2(S,\BZ)\]
which is a primitive curve class. 

Altering the curve class via the isometry $\phi$ therefore results in additional factors of $\frac{d}{m}$ or $\frac{m}{d}$ while keeping the descendent insertions unchanged. This explains the change in exponents
\[ 2g-3 + \deg \longleftrightarrow 2g-3+\degm\] and the factor $m^{\deg - \degm}$ in the multiple cover formula below for the descendent potential. We use the operator~$\bT_{m,\ell}$ introduced in Definition~\ref{def:wrongweighthecke}. As pointed out in Section~\ref{sec:hecke}, this is the $m$-th Hecke operator for functions of weight~$\ell$, which we let act on $\bF_{g,1}$ (which has weight~$2g-12+\degm$). Before stating the conjecture, we want to discuss the role of tautological classes and compatibility with respect to restriction to boundary strata.

\subsection{Compatibility I}\label{sec:compatibilityI}
We will find it convenient to use pullbacks of tautological classes from $\mgn$ instead of $\psi$-classes on $\mgn(S,\beta)$. For $2g-2+n>0$, let 
\[R^*(\mgn)\subseteq A^{*}(\mgn)\] be the tautological ring of $\mgn$.  For a tautological class $\alpha\in R^*(\mgn)$, we consider the invariants
\[ \lan \alpha; \gamma_1,\ldots, \gamma_n \ran =\int_{[\mgn(S,\beta)]^{red}} \pi^*\alpha\cup \prod_{i=1}^n\ev_i^*(\gamma_i)\,, \]
where $\pi\colon \mgn(S,\beta)\to \mgn$ is the stabilization morphism. We write \[\bF_{g,m}\big(\alpha; \gamma_1,\ldots,\gamma_n\big) = \sum_{h\geq 0}\lan \alpha; \gamma_1,\ldots, \gamma_n \ran_{g,mB+hF}\, q^{h-m} \]
for the generating series in divisibility~$m$. By the usual trading of cotangent line classes, these generating series are related to the ones defined via cotangent classes on~$\mgn(S,\beta)$. Any monomial in $\psi$- and $\kappa$-classes can be written, after adding markings, as a product of $\psi$-classes. This procedure leaves $\deg$ and $\degm$ unchanged. Before stating the multiple cover formula below, we explain the compatibility with respect to restriction to boundary strata in $\mgn(S,\beta)$.

A crucial point for this compatibility is the splitting behavior of the reduced class. Consider the pullback of the boundary divisor
\[\Mbar_{g-1,n+2} \to \Mbar_{g,n} \]
under the stabilization morphism $\pi$. Let $\alpha$ be the pushforward of a tautological class (we will omit pushforwards in the notation below). By the restriction property of the reduced class, we obtain
\[ \bF_{g,m}\big(\alpha; \gamma\big) = \bF_{g-1,m}\big(\alpha;\gamma \Delta_S\big)\,.\]
Then, the compatibility follows from two facts. Firstly, for the diagonal class $\Delta_S$ we have
\[ \big(\deg - \degm \big)(\Delta_S) = 0\,,\]
thus the factor $m^{\deg-\degm}$ in Conjecture~\ref{conj:mcf} below remains unchanged. Secondly, we have $\degm(\Delta_S) = 2$ which precisely offsets the genus reduction from $g$ to $g-1$ in the formula
\[ \ell = 2g-2+\degm\,. \]

Next, consider the pullback of the boundary divisor
\[\Mbar_{g_1,n_1+1} \times \Mbar_{g_2,n_2+1} \to \Mbar_{g,n} \]
under the stabilization morphism $\pi$. Let 
\[\alpha = \alpha_1\boxtimes \alpha_2\,,\quad \{1,\ldots,n\}=I_1 \cup I_2\,,\quad \gamma=\gamma_1\boxtimes\gamma_2\] 
be the pushforward of the product of tautological classes, the splitting of markings, and the splitting of the insertions respectively. The Künneth decomposition of the class of the diagonal is denoted by
\[ [\Delta_S] = \sum_j \Delta_j \boxtimes \Delta^j\,.\] The splitting property implies that
\begin{align*} \bF_{g,m}\big(\alpha;\gamma\big) = \sum_{m_1+m_2=m}\sum_{j} \Bigg(&\bF_{g_1,m_1}\big(\alpha_1;\gamma_{I_1}\Delta_j\big) \cdot \bF_{g_2,m_2}^{vir}\big(\alpha_1;\gamma_{I_1}\Delta^j\big) \\
&+\bF_{g_1,m_1}^{vir}\big(\alpha_1;\gamma_{I_1}\Delta_j\big) \cdot \bF_{g_2,m_2}\big(\alpha_1;\gamma_{I_1}\Delta^j\big)\Bigg)\,.\end{align*}
The virtual class for non-zero curve classes vanishes, thus the contribution $\bF^{vir}$ is a number. As a consequence, no non-trivial products of generating series appear when we use boundary expressions. By similar consideration as above, using the $\deg$ and $\degm$ for the diagonal class, we find that the multiple cover formula is compatible with respect to this boundary divisor as well. We can now state the multiple cover formula for the generating series with tautological classes:
\begin{conjecture}\label{conj:mcf}
For $\degm$-homogeneous $\gamma_i\in H^*(S,\BQ)$,  \vspace{2mm}
	\begin{equation*}\vspace{2mm} \bF_{g,m}\big(\alpha;\gamma_1,\ldots,\gamma_n\big) = m^{\deg-\degm}\,  \bT_{m,\ell} \,\Big( \bF_{g,1}\big(\alpha;\gamma_1,\ldots,\gamma_n\big)\Big)\, , \end{equation*} 
where $\deg = \sum \deg(\gamma_i)$, $\degm = \sum \degm(\gamma_i)$ and $\ell = 2g-2+\degm$.
\end{conjecture}
Based on the discussion above, the same formula is conjectured for the potential \[\bF_{g,m}\big(\tau_{a_1}(\gamma_1)\ldots\tau_{a_n}(\gamma_n)\big)\,.\] 

We now show that our presentation of the multiple cover formula is equivalent to the original formula.

\begin{lemma}\label{lem:equiv}
Conjecture~\ref{conj:georgrahul} for all~$d\mid m$ is equivalent to Conjecture~\ref{conj:mcf} for~$m$.
\end{lemma}
\begin{proof}
    By the deformation invariance of the reduced class, the Gromov--Witten invariants for arbitrary curve classes are fully captured by an elliptic K3 surface with a section. The primitive curve classes are $B+hF\in H_2(S,\BZ)$. Taking the coefficient of $q^{mh-m}$ in Conjecture~\ref{conj:mcf} gives a multiple cover formula for the curve class $mB + mhF$ which matches the formula in Conjecture~\ref{conj:georgrahul}. It is the other implication which we have to justify.
    
    The generating series $\bF_{g,m}$ involves curve classes $mB+hF$ of different divisibilities bounded by $m$. We apply Conjecture~\ref{conj:georgrahul} to each invariant and use the isometries $\phi$. Note that each appearance of $\gamma_i = F$ introduces a factor of $\frac{m}{d}$, while each appearance of $\gamma_i = W$ gives $\frac{d}{m}$. Moreover, 
    \[ |\{i \mid \gamma_i = F\}| - |\{i \mid \gamma_i = W\}| = \deg - \degm\,,\]
    and therefore
    \begin{align*} &\bF_{g,m}\big(\alpha;\gamma_1,\ldots,\gamma_n\big) = \sum_{h\geq 0} \big\langle\alpha;\gamma_1,\ldots,\gamma_n\big\rangle_{g,mB+hF}\, q^{h-m} \\
    &= \sum_{h\geq 0}\sum_{\substack{d\mid m\\ d\mid h}}d^{2g-3+\deg} \left(\frac{m}{d}\right)^{\deg -\degm} \big\langle\alpha;\gamma_1,\ldots,\gamma_n\big\rangle_{g,B+\left(\frac{m(h-m)}{d^2}+1\right) F}\, q^{h-m} \\
    &= m^{\deg - \degm} \sum_{d\mid m} d^{2g-3+\degm} \left(\sum_{h\geq 0}\big\langle\alpha;\gamma_1,\ldots,\gamma_n\big\rangle_{g,B+\left(\frac{m}{d}(h-\frac{m}{d})+1\right) F}\, \left(q^d\right)^{h-\frac{m}{d}}\right) \\
    &= m^{\deg - \degm} \sum_{d\mid m} d^{2g-3+\degm} \left(\bB_d \bU_{\frac{m}{d}}\sum_{h\geq 0}\big\langle\alpha;\gamma_1,\ldots,\gamma_n\big\rangle_{g,B+hF}\, q^{h-1} \right)\\
    &= m^{\deg - \degm} \sum_{d\mid m} d^{2g-3+\degm} \, \bB_d \bU_{\frac{m}{d}}\bF_{g,1}\big(\alpha;\gamma_1,\ldots,\gamma_n\big)\\
    &=m^{\deg-\degm}\, \bT_{m,\ell} \,\Big( \bF_{g,1}\big(\alpha;\gamma_1,\ldots,\gamma_n\big)\Big)\,.\qedhere
    \end{align*}
\end{proof}
As a direct consequence, the multiple cover formula implies level~$m$ quasimodularity.
\begin{proposition}\label{prop:mcfimpliesqmod}
If the generating series $\bF_{g,m}$ satisfies the multiple cover formula, it satisfies the quasimodularity conjecture. More precisely,
\[ \bF_{g,m} \in \frac{1}{\Delta(q)^m} \QMod(m)\,.\]
\end{proposition}

\begin{proof}The descendent potentials for primitive curve classes are weakly holomorphic quasimodular with pole of order at most~$1$ and weight~$2g-12+\degm$, see~\cite[Theorem 4]{MPT10} and~\cite[Theorem 9]{BOPY18}. The claim thus follows from Proposition~\ref{prop:level}.
\end{proof}

\subsection{Multiple cover formula in fiber direction}\label{sec:FiberDirection}
When the curve class is a multiple of the fiber class $F$, the multiple cover formula reduces to a property of the Gromov--Witten invariant of elliptic curves. Relevant properties are conjectured in \cite{PixtonBachelor}. 

Let $S\to \BP^1$ be an elliptic K3 surface with section and let $\beta=mF$. By Section~\ref{sec:ProofQmodHAE}, Case 1, we may assume at least one of the insertions is the point class $\gamma_1=\pt$ and $g\geq 1$. Let 
\[ \iota\colon E \hookrightarrow S\]
be the inclusion of a fiber, representing the class $F$. Since the point class is represented by a transverse intersection of $E$ and the section~$B$, the Gromov--Witten theory of $S$ localizes to the Gromov--Witten theory of $E$ with the curve class $mE$. Computation of the obstruction bundle shows that the invariant is of the form
\[\lan \tau_{a_1}(\pt)\tau_{a_2}(\gamma_2)\ldots\tau_{a_n}(\gamma_n)\ran^S_{g,mF}=\lan \lambda_{g-1};\tau_{a_1}(\omega)\tau_{a_2}(\iota^*\gamma_2)\ldots\tau_{a_n}(\iota^*\gamma_n)\ran^E_{g,mE}\]
where $\lambda_{g-1}=c_{g-1}(\BE_g)$. In particular, if $\gamma_i\in \BQ\lan F \ran \oplus U^\perp\oplus \BQ\lan \pt\ran$, the invariant vanishes.
Consider the following generating series \[\bF^E_g\big(\tau_{a_1}(\gamma_1)\ldots\tau_{a_n}(\gamma_n)\big)=\sum_{m\geq 0}\lan \lambda_{g-1}; \tau_{a_1}(\gamma_1)\ldots\tau_{a_n}(\gamma_n)\ran^E_{g,mE}\,q^{m}\] where $\gamma_i=\one$ or $\omega$ and $\sum a_i + \sum \deg(\gamma_i) = g-1+n$.

The generating series $\bF^E_g$ has a simple description in terms of Eisenstein series. The following formula is conjectured in \cite{PixtonBachelor}.
\begin{lemma}\label{lem:EllipticCurve}
For $g\geq 1$, \[\bF^E_g\big(\tau_{g-1}(\omega)\big)= \frac{g!}{2^{g-1}}C_{2g}\,.\]
\end{lemma}
\begin{proof}
    In~\cite[Proposition 4.4.7]{PixtonBachelor} this formula is given under assuming the Virasoro constraint for $\BP^1\times E$. The Virasoro constraint for any toric bundle over a nonsingular variety which satisfies the Virasoro constraint is proven in~\cite{C15}. Combining this result with the Virasoro constraint for elliptic curves \cite{OP06c}, the result follows.
\end{proof}
When $\beta=mF$, Conjecture~\ref{conj:georgrahul} is equivalent to the following proposition.

\begin{proposition}\label{lem:EllipticClosedForm}
There exists $c\in \BQ$ such that \[\bF^E_g\big(\tau_{a_1}(\omega)\ldots\tau_{a_r}(\omega)\tau_{a_{r+1}}(\one)\ldots\tau_{a_{r'}}(\one)\big)=c\, \Dq^{r-1}\bF^E_g\big(\tau_{g-1}(\omega)\big)\,.\]
\end{proposition}
\begin{proof}
     Boundary strata with a vertex of genus less than $g$ do not contribute because the invariants involve $\lambda_h$ vanishes on $\mgn(E,m)$ when $h\geq g$.  If $r'>r$, then $\sum a_i\geq g$ and we can reduce to the case when $r'=r$ by the topological recursion on the $\psi$-monomial in $R^{\geq g}(\mgn)$ \cite{KLLS18}. If $r'=r$, then $\sum a_i=g-1$ and similar argument as in Section~\ref{sec:ProofQmodHAE}, Case $3$ can be applied.  Therefore $\bF^E_g$ is proportional to 
    \[\bF^E_g\big(\tau_{g-1}(\omega)\tau_0(\omega)^{r-1}\big)=\Dq^{r-1}\bF^E_g\big(\tau_{g-1}(\omega)\big)\] where the equality comes from the divisor equation.
\end{proof}
\begin{remark}
 One can find a closed formula for the constant $c\in \BQ$ by integrating tautological classes on $\Mbar_{g,n}$. 
\end{remark}

\section{Holomorphic anomaly equation}\label{sec:HAEm}
This section contains a proof of Proposition~\ref{prop:mcfimplieshae}. We derive the holomorphic anomaly equation for $m\geq 1$ from the conjectural multiple cover formula, such that both are compatible\footnote{We should point out that this derivation should be lifted to the cycle-valued holomorphic anomaly equation. Tautological classes play no role here.}. It turns out that the equation is almost identical to the one in the primitive case. Additional factors appear only in the last two terms, which are specific to K3 surfaces. We refer to~\cite[Section 7.3]{OP19} for explanations on the appearance of these terms. 
\begin{proof}[Proof of Proposition~\ref{prop:mcfimplieshae}]
Let $\gamma_1,\ldots,\gamma_n\in H^*(S)$ with
\[ \deg = \sum_i \deg(\gamma_i)\,,\quad \degm= \sum_i\degm(\gamma_i)\,. \]
We will simply write $\gamma$ to denote $\gamma_1,\ldots,\gamma_n$.
Assume that the multiple cover formula~\eqref{eq:mcf} holds for all divisors $d\mid m$ and all descendent insertions. Using Lemma~\ref{lem:equiv}, also Conjecture~\ref{conj:mcf} holds. By Proposition~\ref{prop:mcfimpliesqmod}, the descendent potentials are quasimodular forms of level~$m$ and we can consider the $\dc$-derivative. We apply the $\dc$-derivative to Conjecture~\ref{conj:mcf} and use the commutator relations Lemma
~\ref{lem:operators} to obtain:
\begin{align*}
    \dc \bF_{g,m}\big(\alpha;\gamma\big) &= \dc \Big(m^{\deg - \degm} \bT_{m,2g-2+\degm} \bF_{g,1}\big(\alpha;\gamma\big)\Big) \\
    &=m^{\deg - \degm+1} \bT_{m,2g-4+\degm} \dc\bF_{g,1}\big(\alpha;\gamma\big)\,.
\end{align*}
We want to explain that the last row precisely recovers the definition of $\bH_{g,m}$ in~\eqref{eq:fullHAE}, after applying the holomorphic anomaly equation for the primitive series~\cite[Theorem 4]{OP18}:
\[ \dc\bF_{g,1}\big(\alpha; \gamma\big) = \bH_{g,1}\big(\alpha; \gamma\big)\,.\]
We do so by explaining how each term of $\bH_{g,1}\big(\alpha; \gamma\big)$ is affected:
\begin{enumerate}[label=(\roman*)]
\item The degree $\deg$ of $\bF_{g-1,1}\big(\alpha;\gamma \, \Delta_{\BP^1}\big)$ has increased by one. The genus, however, dropped by~$1$. Thus, the first term precisely matches the multiple cover formula, i.e.\ 
\[ \bF_{g-1,m}\big(\alpha; \gamma\Delta_{\BP^1}\big)=m^{\deg - \degm+1} \bT_{m,2g-4+\degm}\Big(\bF_{g-1,1}\big(\alpha;\gamma \,\Delta_{\BP^1}\big)\Big)\,. \]
\item The virtual class is non-zero only for curve class $\beta=0$ and genus $0$, $1$, see Section~\ref{sec:introduction}. In these cases, the potential $\bF_{g_2}^{vir}$ is simply a number and the operator $\bT_{m,\ell}$ acts non-trivially only on $\bF_{g_1,m}$. We distinguish the two cases:\medskip

\noindent $g_2 = 0$. The virtual class is given by the fundamental class and the integral is given by intersection pairing on $S$. Non-trivial terms are obtained from $\delta_i^\vee=1$ or $F$. If $\delta_i^\vee=1$ then 
\[ \deg(\gamma_{I_2}) = \degm(\gamma_{I_2} )= 2\,.\]
The modified degree $\degm$ of $\bF_{g_1,1}\big(\alpha_{I_1}; \gamma_{I_1} \delta_i\big)$ has decreased by $2$, whereas $\deg$ decreased by $1$ (the insertion $\delta_i = F$ contributes $\deg = 1$). The term thus matches the multiple cover formula:
\begin{align*} &\bF_{g_1,m}\big(\alpha_{I_1}; \gamma_{I_1} \delta_i\big)\\
&= m^{\deg - \degm+1} \bT_{m,2g-4+\degm} \Big(\bF_{g_1,1}\big(\alpha_{I_1}; \gamma_{I_1} \delta_i\big)\Big)\,. \end{align*}
If $\delta_i^\vee=F$ then 
\[ \deg(\gamma_{I_2}) = 1\,,\quad \degm(\gamma_{I_2} )= 2\,.\]
The modified degree $\degm$ of $\bF_{g_1,1}\big(\alpha_{I_1}; \gamma_{I_1} \delta_i\big)$ has decreased by $2$, whereas $\deg$ decreased by $1$. The term matches the multiple cover formula.
\medskip

\noindent $g_2=1$. The virtual class is given by $c_2(S)$ and the integral is given by intersection pairing on $S$. Non-trivial terms are obtained only from $\delta_i^\vee=1$ and 
\[ \deg(\gamma_{I_2}) = \degm(\gamma_{I_2} )= 0\,.\]
Analogously to case (i), the degree $\deg$ of $\bF_{g_1,1}\big(\alpha_{I_1}; \gamma_{I_1} \delta_i\big)$ has increased by $1$, $\degm$ remained unchanged, and the genus dropped by $1$. The term matches the multiple cover formula. 
\item The modified degree $\degm$ of $\bF_{g,1}\big(\alpha \psi_i; \gamma_1,\ldots,\pi^*\pi_*\gamma_i,\ldots,\gamma_n \big)$ has decreased by $2$, whereas $\deg$ decreased by $1$. Again we find that the term matches the multiple cover formula
\begin{align*} &\bF_{g,m}\big(\alpha \psi_i; \gamma_1,\ldots,\pi^*\pi_*\gamma_i,\ldots,\gamma_n \big)\\
&= m^{\deg - \degm+1} \bT_{m,2g-4+\degm} \Big(\bF_{g,1}\big(\alpha \psi_i; \gamma_1,\ldots,\pi^*\pi_*\gamma_i,\ldots,\gamma_n \big)\Big)\,. \end{align*}
\item The degree of $\langle \gamma_i, F \rangle \bF_{g,1}\big(\alpha;\gamma_1,\ldots,F,\ldots,\gamma_n\big)$ remains unchanged, whereas $\degm$ decreased by $2$. An additional factor of $\frac{1}{m}$ therefore appears:
\begin{align*}
    &\frac{1}{m}\langle \gamma_i, F \rangle \bF_{g,m}\big(\alpha;\gamma_1,\ldots,F,\ldots,\gamma_n\big)\\
    &= m^{\deg - \degm+1} \bT_{m,2g-4+\degm}\Big( \langle \gamma_i, F \rangle \bF_{g,1}\big(\alpha;\gamma_1,\ldots,F,\ldots,\gamma_n\big)\Big)\,.
\end{align*}
\item The term $\bF_{g,1}\big(\ldots, \sigma_1(\gamma_i,\gamma_j),\ldots,\sigma_2(\gamma_i,\gamma_j),\ldots\big)$ is similar to the previous case: $\deg$ remains unchanged, whereas $\degm$ decreases by $2$, giving rise to an additional factor of $\frac{1}{m}$:
\begin{align*} &\frac{1}{m} \bF_{g,m}\big(\gamma_1,\ldots, \sigma_1(\gamma_i,\gamma_j),\ldots,\sigma_2(\gamma_i,\gamma_j),\ldots,\gamma_n\big) \\
&=m^{\deg - \degm+1} \bT_{m,2g-4+\degm}\Big(\bF_{g,1}\big(\gamma_1,\ldots, \sigma_1(\gamma_i,\gamma_j),\ldots,\sigma_2(\gamma_i,\gamma_j),\ldots,\gamma_n\big)\Big)
\end{align*}
\end{enumerate}
We arrive at the level $m$ holomorphic anomaly equation~\eqref{eq:fullHAE} which appeared in Section~\ref{sec:introduction}.
\end{proof}
\subsection{Divisor equation}

For primitive curve classes, it was pointed out in~\cite[Section 3.6, Case (i)]{OP18} that the holomorphic anomaly equation in genus $0$ is compatible with the divisor equation. For divisibility $m$, let
\[\frac{d}{d\gamma} = \langle \gamma, F\rangle \Dq + m\langle \gamma, W\rangle\,, \hspace{3mm} \gamma\in H^2(S)\,. \]
The divisor equation implies that
\begin{align*}&\bF_{g,m}\big(\tau_{a_1}(\gamma_1)\ldots\tau_{a_{n-1}}(\gamma_{n-1})\tau_0(\gamma_n)\big) \\&=\frac{d}{d\gamma_n}\bF_{g,m}\big(\tau_{a_1}(\gamma_1)\ldots\tau_{a_{n-1}}(\gamma_{n-1})\big)\\
&\mi +\sum_{i=1}^{n-1} \bF_{g,m}\big(\tau_{a_1}(\gamma_1)\ldots\tau_{a_i-1}(\gamma_i\cup \gamma_n)\ldots\tau_{a_{n-1}}(\gamma_{n-1})\big)\,.
\end{align*}
The compatibility with the divisor equation corresponds to 
\begin{align}
    &\bH_{g,m}\big(\tau_{a_1}(\gamma_1)\ldots\tau_{a_{n-1}}(\gamma_{n-1})\tau_0(\gamma_n)\big)\label{eq:divisoreq}\\
    &= \frac{d}{d\gamma_n}\bH_{g,m} \big(\tau_{a_1}(\gamma_1)\ldots\tau_{a_{n-1}}(\gamma_{n-1})\big) \nonumber\\
    &\mi -2k \bF_{g,m}\big(\tau_{a_1}(\gamma_1)\ldots\tau_{a_{n-1}}(\gamma_{n-1})\big) \nonumber\\
    & \mi +\sum_{i=1}^{n-1} \bH_{g,m}\big(\tau_{a_1}(\gamma_1)\ldots\tau_{a_i-1}(\gamma_i\cup \gamma_n)\ldots\tau_{a_{n-1}}(\gamma_{n-1})\big)\,,\nonumber
\end{align}
    where $k$ is the weight of $\bF_{g,m}\big(\tau_{a_1}(\gamma_1)\ldots\tau_{a_{n-1}}(\gamma_{n-1})\big)$ and we have used the commutator relation
\[ \Big[ \dc,\Dq \Big] = -2k\,. \]

The same check as in the primitive case works for arbitrary divisibility. This relies on the fact that the divisor equation for $W$ is the same as applying the differential operator
\[ \Dq = q \frac{d}{dq}\]
to the generating series. Indeed, for the curve class $\beta = mB+ hF$,
\[\langle \beta,W \rangle = -2m +h +m = h-m\,, \]
which matches the exponent of $q^{h-m}$ in the generating series~$\bF_{g,m}$.
The divisor equation for $F$ acts as multiplication by $m$ on the generating series.

In Section~\ref{sec:ProofQmodHAE}, the refined induction reduces any generating series ultimately to genus $0$ and $1$. We thus have to justify compatibility of the holomorphic anomaly equation for generating series of the form
\[ \bF_{1,m} \big( \tau_0(\pt) \tau_0(\gamma_1)\ldots\tau_0(\gamma_n)\big)\,,\quad \gamma_i\in H^2(S)\,. \]
This compatibility however is true. By Proposition~\ref{prop:initial}, the multiple cover formula, which is compatible with the divisor equation, holds in genus $\leq 1$. Thus, we also find compatibility for the holomorphic anomaly equation.

\begin{example}
We consider $\bF_{0,m}\big(\tau_0(W)^2\big)$ to illustrate the above compatibility. To compute $\bH_{0,m}$, we use that $\sigma(W\boxtimes W) = U^\perp$, where the endomorphism~$\sigma$ is as defined in Section~\ref{sec:introduction}.  Since the curve classes are contained in $U$, application of the divisor equation to a basis of $U^\perp$ implies
\[\bF_{0,m} \big( \tau_0(U^\perp)\big) = 0\,. \]
We find that
\[ \bH_{0,m}\big(\tau_0(W)^2\big) = -4 \bF_{0,m}\big(\tau_1(\one)\tau_0(W)\big) + \frac{40}{m} \bF_{0,m}\big(\tau_0(F)\tau_0(W)\big)\,.\]
In the above notation, $\gamma_n = W$ is the second $W$ and $k=-10$ is the weight of $\bF_{0,m}\big(\tau_0(W)\big)$. We have to check that
\[ \bH_{0,m}\big(\tau_0(W)^2\big) = \Dq\bH_{0,m}\big(\tau_0(W)\big) +20 \bF_{0,m}\big(\tau_0(W)\big)\,. \]
By the dilaton equation, we can verify
\begin{align*} &\bH_{0,m}\big(\tau_0(W)^2\big) - \Dq \bH_{0,m}\big(\tau_0(W)\big)\\
&= -2 \Dq \bF_{0,m}\big(\tau_1(\one)\big) -4 \bF_{0,m}\big(\tau_0(W)\big) +\frac{20}{m} \bF_{0,m}\big(\tau_0(F)\tau_0(W)\big) \\
&= 4 \Dq \bF_{0,m}\big(\emptyset\big) - 4 \Dq \bF_{0,m}\big(\emptyset\big) +20\bF_{0,m}\big(\tau_0(W)\big) \\
&=20 \bF_{0,m}\big(\tau_0(W)\big)\,. \end{align*}
\end{example}

\begin{example}\label{ex:HAEm}
The above example in genus~$0$ illustrates how the second last term in the holomorphic anomaly equation (\ref{eq:hae}) plays a role. We consider
\[ \bF_{1,m} \big( \tau_1(W) \tau_0(W)\big) \]
to show how the last term, i.e.\ the term involving $\sigma$, interacts non-trivially with the other terms. 
The corresponding series $\bH_{1,m}$ are
\begin{align*} \bH_{1,m}\big( \tau_1(W)\tau_0(W)\big) &= 2\bF_{0,m}\big(\tau_1(W)\tau_0(W)\tau_0(\one)\tau_0(F)\big) \\
&\mi-2\Big( \bF_{1,m}\big( \tau_2(\one)\tau_0(W)\big) + \bF_{1,m}\big(\tau_1(W)\tau_1(\one)\big) \Big) \\
&\mi+\frac{20}{m} \Big( \bF_{1,m}\big( \tau_1(F)\tau_0(W)\big) + \bF_{1,m}\big( \tau_1(W)\tau_0(F)\big) \Big) \\
&\mi-\frac{2}{m} \bF_{1,m}\big( \psi_1; \Delta_{U^\perp} \big)\,,\\
\bH_{1,m}\big( \tau_1(W) \big) &= 2 \bF_{0,m}\big(\tau_1(W)\tau_0(\one)\tau_0(F)\big) \\
&\mi- 2 \bF_{1,m}\big(\tau_2(\one)\big) \\
&\mi+\frac{20}{m} \bF_{1,m}\big(\tau_1(F)\big)\,. \end{align*}
Let $k=-8$ be the weight of $\bF_{1,m}\big(\tau_1(W)\big)$. Then~\eqref{eq:divisoreq} is equivalent to
\[ \bH_{1,m}\big(\tau_1(W)\tau_0(W)\big) = \Dq \bH_{1,m}\big(\tau_1(W) \big) -2k \bF_{1,m}\big(\tau_1(W)\big)\,.\]
The term $\bF_{1,m}\big( \psi_1; \Delta_{U^\perp} \big)$ can be computed using
\[ \psi_1 = [\delta_1]+\frac{1}{24} [\delta_0] \in A^1(\Mbar_{1,2})\,, \]
where $[\delta_0]\in A^1(\Mbar_{1,2})$ is the class of the pushforward of the fundamental class under the map
\[ \Mbar_{0,4}\to \Mbar_{1,2}\]
gluing the third and fourth markings and $[\delta_1]$ is the class of the boundary divisor of curves with a rational component carrying both markings. The genus~$0$ contribution vanishes by the divisor equation. Since the rank of $U^\perp$ is $20$, we obtain the genus~$1$ contribution
\[ \bF_{1,m}\big( \psi_1; \Delta_{U^\perp} \big) = 20 \bF_{1,m} \big( \tau_0(\pt)\big)\,. \]
The divisor equation for $F$ implies that
\[ \frac{20}{m}\bF_{1,m} \big(\tau_1(W)\tau_0(F) \big ) = 20\bF_{1,m}\big(\tau_1(W)\big) + \frac{20}{m} \bF_{1,m}\big(\tau_0(\pt)\big)\,. \]
We can now verify the compatibility by a direct computation using divisor and dilaton equation:
\begin{align*}
\bH_{1,m}\big( \tau_1(W) \tau_0(W)\big) &= \Dq \bH_{1,m}\big( \tau_1(W)\big) -2 \bF_{1,m}\big(\tau_1(W)\big) -2 \bF_{1,m}\big(\tau_1(W)\tau_1(\one)\big) \\
&\mi+\frac{20}{m} \bF_{1,m}\big(\tau_0(\pt) \big) + \frac{20}{m} \bF_{1,m}\big( \tau_1(W)\tau_0(F) \big) \\
&\mi-\frac{2}{m}\bF_{1,m}\big(\psi_1; \Delta_{U^\perp}\big) \\
&= \Dq \bH_{1,m}\big( \tau_1(W)\big)-4 \bF_{1,m}\big(\tau_1(W)\big)\\
&\mi+\frac{20}{m} \bF_{1,m}\big(\tau_0(\pt)\big)+ \frac{20}{m} \bF_{1,m}\big( \tau_1(W)\tau_0(F) \big) \\
&\mi-\frac{40}{m} \bF_{1,m} \big( \tau_0(\pt)\big) \\
&=\Dq \bH_{1,m}\big( \tau_1(W)\big) + 16\bF_{1,m}\big(\tau_1(W)\big)\,. \end{align*}
\end{example}

\section{Relative holomorphic anomaly equation}\label{sec:RelativeHAE}

In this section, we first state the degeneration formula for the reduced virtual class under the degeneration to the normal cone. For primitive curve class, the formula is proven in \cite{MPT10}. For sake of completeness, we summarize a proof for arbitrary divisibility in Appendix~\ref{app:degeneration}. Then, we state the relative holomorphic anomaly equation and prove the compatibility with the degeneration formula.

\subsection{Degeneration formula}
Let $S\to \BP^1$ be an elliptic K3 surface with a section. For $m\geq 1$, let $\beta =mB+hF$ be a curve class.
Choose a smooth fiber $E$ of $S\to \BP^1$. Let $\epsilon\colon \CS\to \BA^1$ be the total space of the degeneration to the normal cone of $E$ in $S$. This space corresponds to the degeneration 
\begin{equation} \label{pic:degen}
    S \rightsquigarrow S\cup_E \BP^1\times E \,.
\end{equation}
Over the center $\iota: 0 \hookrightarrow \BA^1$, the fiber is $S\cup_E \BP^1\times E$ and over $t\neq 0$, the fiber is isomorphic to $S$. Let $\mgn(\epsilon,\beta)$ be the moduli space of stable maps to the degeneration $\CS$. Over $t\neq 0$, this moduli space is isomorphic to $\mgn(S,\beta)$ and over $t=0$, this moduli space parametrizes stable maps to the expanded target \[\widetilde{\CS}_0 = S\cup_E \BP^1\times E \cup_E \cdots \cup_E  \BP^1\times E\,.\]
Let $$\nu = (g_1,g_2,n_1,n_2,h_1,h_2)$$ be a splitting of the discrete data $g,n,h$ and let $\beta_i = mB+h_iF$ be the splitting of the curve class. An ordered partition of $m$ \[\mu = (\mu_1,\ldots,\mu_l)\] specifies the contact order along the relative divisor $E$. 

Let $l = \mathrm{length}(\mu)$ and $\Mbar_{g,n}(\CS_0,\nu)_{\mu}$ be the fiber product 
\begin{equation}
    \Mbar_{g,n}(\CS_0,\nu)_{\mu} = \Mbar_{g_1,n_1}(S/E,\beta_1)_\mu \times_{E^l} \Mbar^{\bullet}_{g_2,n_2}(\BP^1\times E/E,\beta_2)_{\mu}
\end{equation}
of the boundary evaluations at relative markings\footnote{We put $^\bullet$ to indicate (possibly) disconnected theory. Namely, for each connected component $C$ of the domain curve, intersection of $C$ with the relative divisor $E$ is nontrivial.} and let 
\begin{equation*}
    \iota_{\nu\mu}\colon \Mbar_{g,n}(\CS_0,\nu)_{\mu} \to \Mbar_{g,n}(\CS_0,\beta) 
\end{equation*}
be the finite morphism. Let $\Delta_{E^l}\colon E^l \to E^l\times E^l$ be the diagonal embedding.


\begin{theorem}\label{thm:degeneration}
The reduced virtual class of maps to the degeneration~\eqref{pic:degen} satisfies the following properties.
\begin{enumerate}[label=(\roman*)]
    \item For $\iota_t\colon \{t\}\hookrightarrow \BA^1$, the Gysin pullback of reduced class is given by
    \[\iota_t^![\Mbar_{g,n}(\epsilon,\beta)]^{red} = [\Mbar_{g,n}(\CS_{t},\beta)]^{red} \,.\]
    \item For the special fiber, 
    \[ [\Mbar_{g,n}(\CS_0, \beta)]^{red} = \sum_{\nu,\mu} \frac{\prod_i \mu_i}{l!}\iota_{\nu\mu*}[\Mbar_{g,n}(\CS_0,\nu)_{\mu}]^{red} \,.\]
    \item On the special fiber, we have the factorization 
    \begin{align*}
    [\Mbar_{g,n}(\CS_0,\nu)_{\mu}]^{red} = \Delta_{E^l}^!\Big(&[\Mbar_{g_1,n_1}(S/E,\beta_1)_{\mu}]^{red}\\
     &\mi \times [\Mbar^\bullet_{g_2,n_2}(\BP^1\times E/E,\beta_2)_{\mu}]^{vir} \Big) \,. 
\end{align*}
\end{enumerate}
\end{theorem}
\begin{proof}
When $m\geq 1$, the reduced class of the disconnected moduli space $\mgn^\bullet(S/E,\beta)$ vanishes on all components parameterizing maps with at least two connected components. Therefore, disconnected theory can only appear on the bubble $\BP^1\times E$. The proof is given in Appendix~\ref{app:degeneration}. 
\end{proof}
Denote an ordered cohomology weighted partition by \[\unmu = \big( (\mu_1,\delta_1), \ldots, (\mu_l,\delta_l) \big)\,,\; \delta_i\in H^*(E)\]
 and let $\omega \in H^2(E)$ be the point class. The descendent potential for the pair $(S,E)$ is defined analogously to the absolute case:
\[ \bF_{g,m}^{\rel}\big(\alpha;\gamma_1,\ldots,\gamma_n\mid\unmu\big) = \sum_{h\geq 0} \big\langle \alpha;\gamma_1,\ldots,\gamma_n\mid\unmu\big\rangle_{g,mB+hF}^{S/E}\, q^{h-m}\,.\] 
The descendent potential for the pair $(\BP^1\times E,E)$ is defined by
\[\bG^{\rel,\bullet}_{g,m}\big(\alpha;\gamma_1,\ldots,\gamma_n\mid\unmu \big) = \sum_{h\geq 0} \lan \alpha;\gamma_1,\ldots,\gamma_n\mid\unmu\ran_{g,mB+hF}^{\BP^1\times E/E,\bullet}\; q^h\,.\]
As a corollary, we get the degeneration formula of reduced Gromov--Witten invariants.
\begin{corollary}\label{cor:degeneration} 
Let $\gamma_1,\ldots,\gamma_n\in H^*(S)$ and choose a lift of these cohomology classes to the total space $\CS$. Then
\begin{align}
\bF_{g,m}\big(\tau_{a_1}(\gamma_1)\ldots\tau_{a_n}(\gamma_n)\big) = \sum_{\nu}\sum_{\unmu\neq \underline{\mu_\omega}} \frac{\prod_i \mu_i}{l!} \bF^{\rel}_{g_1,m}\big(\ldots\mid\unmu\big) \cdot\bG^{\rel,\bullet}_{g_2,m}\big(\ldots\mid\unmu^\vee\big)\,,   
\end{align}
where  \[\unmu^\vee = \big( (\mu_1,\delta_1^\vee), \ldots, (\mu_l,\delta_l^\vee) \big) \text{ and } \underline{\mu_\omega}=\big( (\mu_1,\omega), \ldots, (\mu_l,\omega)\big)\,.\]
\end{corollary}
\begin{proof}
By Theorem~\ref{thm:degeneration}, we are left to prove that the relative profile $\underline{\mu_\omega}$ on $S/E$ has vanishing contribution. Let $x$ be the intersection of the section of the elliptic fibration and the fiber $E$. We consider $(E,x)$ as an abelian variety. Let $K$ be the kernel of the following morphism between abelian varieties
\begin{equation*}
    E^l \to \mathsf{Pic}^0(E)\,,(x_i)_i \mapsto \CO_E\Big(\sum_i \mu_i( x_i-x)\Big) \,.
\end{equation*}
Consider a stable map $f$ from a curve $C$ to an expanded degeneration of $S/E$. The equality $f_*[C] = \beta_1$ (after pushforward to $S$) in $H_2(S,\mathbb{Z})$ lifts to a rational equivalence of line bundles on $S$ because the cycle-class map $$c_1 \colon \mathsf{Pic}(S)\to H^2(S,\mathbb{Z})\cong H_2(S,\mathbb{Z})$$ is injective. Intersecting with the relative divisor, the two line bundles are, respectively, $\CO_E(\sum \mu_i x_i)$ and $\CO_E(mx)$. Thus, we see that the evaluation map $\Mbar_{g_1,n_1}(S/E,\beta_1)\to E^l$ factors through $K$. Since $K\subset E^l$ has codimension~$1$ a generic point on $E^l$ does not lie on~$K$ and thus the contribution from the relative profile $\underline{\mu_\omega}$ vanishes.
\end{proof}


\subsection{Relative holomorphic anomaly equations}\label{sec:degenerationHAE}

Assuming quasimodularity, we have two ways to compute the derivative of $\bF_{g,m}$ with respect to $C_2$: 
\begin{enumerate}[label=(\roman*)]
\item Apply the degeneration formula Corollary~\ref{cor:degeneration}, together with the holomorphic anomaly equations for $(S,E)$ and $(\BP^1\times E,E)$.
\item Apply the holomorphic anomaly equation~\eqref{eq:hae2} for $S$, followed by the degeneration formula for each term.
\end{enumerate}
We argue that both ways yield the same result. This compatibility is parallel to the compatibility proved in~\cite[Section 4.6]{OP19}. We first state the holomorphic anomaly equations for the relevant relative geometries.
\subsection*{Relative $(\BP^1\times E,E)$} Consider $\pi\colon \BP^1\times E\to \BP^1$ as a trivial elliptic fibration over $\BP^1$. For the pair $(\BP^1\times E,E)$ the holomorphic anomaly equation holds for cycle-valued generating series~\cite{OP19}. The equation for descendent potentials can thus be obtained by integrating against tautological classes $\alpha\in R^*(\mgn)$. For insertions $\gamma_i\in H^*(\BP^1\times E,\BQ)$ we will simply write $\gamma$. Let $\unmu = \big( (\mu_1,\delta_1), \ldots, (\mu_l,\delta_l) \big)$ and $\unmu'$ be ordered cohomology weighted partitions. We denote by 
\[\bG^{\sim,\bullet}_{g,m}\big(\unmu\mid\alpha;\gamma\mid\unmu'\big)=\sum_{h\geq 0}\lan \unmu\mid\alpha;\gamma\mid\unmu'\ran^{\BP^1\times E,\sim,\bullet}_{g,m\BP^1+hE}\; q^h\]
the disconnected rubber generating series for $\BP^1\times E$ relative to divisors at $0$ and $\infty$. Let $\Delta_E\subset E\times E$ be the class of the diagonal. Define the generating series 
\begin{align*} &\bP^{\rel,\bullet}_{g,m} \big(\alpha;\gamma \mid \unmu \big) \\
&= \bG^{\rel,\bullet}_{g-1,m} \big(\alpha;\gamma,\Delta_{\BP^1} \mid \unmu \big)\\ &\mi+ 2 \sum_{\substack{g=g_1+g_2\\ \{1,\ldots,n\}=I_1\sqcup I_2\\\forall i\in I_2:\gamma_i\in H^2(E)\\ h\geq 0}} \sum_{\substack{b;b_1,\ldots,b_h \\ l_1,\ldots,l_h}} \frac{\prod_{i=1}^h b_i}{h!} \bG^{\rel,\bullet}_{g_1,m} \big(\alpha_{I_1};\gamma_{I_1}\mid  ((b,\one), (b_i, \Delta_{E,{\ell_i}})_{i=1}^h)\big) \\ &\hspace{5cm} \times \bG^{\sim,\bullet}_{g_2,m} \big( ((b,\one), (b_i, \Delta_{E,{\ell_i}}^\vee)_{i=1}^h) \mid \alpha_{I_2};\gamma_{I_2} \mid \unmu \big) \\ &\mi-2 \sum_{i=1}^n  \bG^{\rel,\bullet}_{g,m}\big(\alpha\psi_i;\gamma_1,\ldots,\gamma_{i-1},\pi^*\pi_*\gamma_i,\gamma_{i+1},\ldots,\gamma_n \mid \unmu\big) \\ &\mi-2 \sum_{i=1}^l  \bG^{\rel,\bullet}_{g,m}\big(\alpha;\gamma \mid (\mu_1,\delta_1),\ldots,(\mu_i, \psi_i^{\rel}\pi^*\pi_* \delta_i),\ldots,(\mu_l,\delta_l) \big) \end{align*}
where $\psi_i^{\rel}$ is the cotangent line class at the $i$-th relative marking and $\Delta_E=\sum \Delta_{E,l_i}\otimes \Delta_{E,l_i}^\vee$ is the pullback of the K\"unneth decomposition of $\Delta_E$ at the corresponding relative marking.
The holomorphic anomaly equation takes the form:
\begin{proposition}(\cite[Proposition 20]{OP19})
$\bG^{\rel,\bullet}_{g,m}(\alpha;\gamma\mid\unmu)$ is a quasimodular form and
\[\dc \bG^{\rel,\bullet}_{g,m}(\alpha;\gamma\mid\unmu) = \bP^{\rel,\bullet}_{g,m}(\alpha;\gamma\mid\unmu)\,.\]
\end{proposition}
\subsection*{Relative $(S,E)$}
Since the log canonical bundle of $(S,E)$ is nontrivial, relative moduli spaces in fiber direction have nontrivial virtual fundamental class. Define
\[\bF^{vir-\rel}_{g,0}(\alpha;\gamma\mid\emptyset)=\sum_{h\geq 0}\lan\alpha;\gamma\mid\emptyset\ran^{S/E,vir}_{g,hF}\, q^h\,.\]
Recall that we denote the pullback of the diagonal of $\BP^1$ as \[\Delta_{\BP^1} = 1\boxtimes F + F \boxtimes 1 = \sum_{i=1}^2 \delta_i\boxtimes \delta_i^\vee\,.\] Define a generating series
\begin{align*}
    &\bH^{\rel}_{g,m}\big(\alpha; \gamma \mid \unmu \big) \\
    &= \bF^{\rel}_{g-1,m}\big(\alpha;\gamma,\Delta_{\BP^1}\mid \unmu\big) \\
    &\mi+2\sum_{\substack{g=g_1+g_2\\ \{1,\ldots,n\}=I_1\sqcup I_2\\ i\in\{1,2\}}} \bF^{\rel}_{g_1,m}\big(\alpha_{I_1};\gamma_{I_1},\delta_i\mid \unmu\big)\,\bF^{vir-\rel}_{g_2, 0}\big(\alpha_{I_2};\gamma_{I_2},\delta_i^{\vee}\mid \emptyset\big) \\
    &\mi+2\sum_{\substack{g=g_1+g_2\\\{1,\ldots,n\}=I_1\sqcup I_2\\ \forall i\in I_2:\gamma_i\in H^2(E)\\ h\geq 0}} \sum_{\substack{b;b_1,\ldots,b_h \\ l_1,\ldots,l_h}} \frac{\prod_{i=1}^h b_i}{h!} \bF^{\rel}_{g_1,m} \big(\alpha_{I_1}; \gamma_{I_1}\mid ((b,\one),(b_i, \Delta_{E,{\ell_i}})_{i=1}^h)\big) \\ &\hspace{5cm} \times \bG^{\sim,\bullet}_{g_2,m} \big( ((b,\one), (b_i, \Delta_{E,{\ell_i}}^\vee)_{i=1}^h) \mid \alpha_{I_2}; \gamma_{I_2} \mid \unmu \big)  \\&\mi-2\sum_{i=1}^n \bF^{\rel}_{g,m}\big(\alpha\psi_i;\gamma_1,\ldots,\gamma_{i-1},\pi^*\pi_*\gamma_i,\gamma_{i+1},\ldots,\gamma_n\mid \unmu\big)\\
    &\mi-2 \sum_{i=1}^l  \bF^{\rel}_{g,m}\big(\alpha; \gamma \mid ((\mu_1,\delta_1),\ldots,(\mu_i, \psi_i^{\rel}\pi^*\pi_* \delta_i),\ldots,(\mu_l,\delta_l)) \big)\\
    &\mi+\frac{20}{m}\sum_{i=1}^n\langle\gamma_i,F\rangle\bF^{\rel}_{g,m}\big(\alpha;\gamma_1,\ldots,\gamma_{i-1},F,\gamma_{i+1},\ldots,\gamma_n\mid \unmu\big)\\
    &\mi-\frac{2}{m}\sum_{i<j}\bF^{\rel}_{g,m}\big(\alpha;\gamma_1,\ldots,\underbrace{\sigma_1(\gamma_i,\gamma_j)}_{i \textup{th}},\ldots,\underbrace{\sigma_2(\gamma_i,\gamma_j)}_{j \textup{th}},\ldots,\gamma_n\mid \unmu\big)\,.
\end{align*}
The conjectural holomorphic anomaly equation for $(S,E)$ has the following form: \[\bF^{\rel}_{g,m}(\alpha;\gamma\mid\unmu) \in \frac{1}{\Delta(q)^m}\QMod(m)\] and
\begin{equation}
    \dc\bF^{\rel}_{g,m}(\alpha;\gamma\mid\unmu)=\bH^{\rel}_{g,m}(\alpha;\gamma\mid\unmu)\,.
\end{equation}
\begin{proposition}\label{prop:compatibility}
Let $m\geq 1$. Assuming quasimodularity for $\bF_{g,m}$ and $\bF_{g,m}^{\rel}$, the holomorphic anomaly equations are compatible with the degeneration formula in the above sense.
\end{proposition}
\begin{proof} The proof given in \cite[Proposition 21]{OP19} treats virtual fundamental classes, not reduced classes. The splitting behavior of the reduced class with respect to restriction to boundary divisors~\cite[Section 7.3]{MPT10} calls for a slight adaptation of the proof. For this, we introduce a formal variable $\varepsilon$ with $\varepsilon^2 = 0$. We can then interpret reduced Gromov--Witten invariants of the K3 surface as integrals against the class\footnote{We thank G.\ Oberdieck for pointing this out.}
\[ [\mgn(S,\beta)]^{vir} + \varepsilon \,[\mgn(S,\beta)]^{red}\]
followed by taking the $[\varepsilon]$-coefficient. We consider a similar class for $S/E$. This class has the advantage of satisfying the usual splitting behavior of virtual fundamental classes.  Thus, for this class one can follow the proof of compatibility given in~\cite[Proposition 21]{OP19}. All the terms appearing in the computation~(ii) also appear in computation~(i). We are left with proving the cancellation of the remaining terms in~(i). This follows from comparing $\psi^{\rel}_i$-class and the $\psi$-class pulled-back from the stack of target degeneration~\cite[Lemma 22]{OP19}. In particular, we match the following terms: the third term of $\bH^{\rel}$ times $\bG^{\rel,\bullet}$ with the fourth term of $\bF^{\rel}$ times $\bP^{\rel,\bullet}$; and analogously for the fifth term of $\bH^{\rel}$ times $\bG^{\rel,\bullet}$ with the second term of $\bF^{\rel}$ times $\bP^{\rel,\bullet}$.
\end{proof}
The main advantage of the holomorphic anomaly equation is that it is compatible with the degeneration formula. Thus, the genus reduction from the degeneration formula connects the low genus results with arbitrary genus predictions. On the other hand, it is not even clear to say what should be the compatibility of the multiple cover formula and the degeneration formula.

\section{Tautological relations and initial condition}\label{sec:InitialCondition}

This section contains a proof of the multiple cover formula in genus~$0$ and genus~$1$ for any divisibility~$m$. It is a direct consequence of the KKV formula. However, as initial condition for our induction we also require a special case in genus~$2$, which cannot be easily deduced from the KKV formula. We treat this descendent potential separately, using double ramification relations~\cite{BHPSS20} for K3 surfaces. This approach is likely to give \textit{relations in any genus} and will be pursued in the future.


\subsection{Double ramification relations}\label{sec:DRrelation}
In this section we recall {\em double ramification relations with target variety} developed in~\cite{B20,BHPSS20}.

Let $\pic$ be the Picard stack for the universal curve over the stack of prestable curves $\mathfrak{M}_{g,n}$ of genus $g$ with $n$ markings. Let \begin{equation}\label{eq:tautStrpic}
    \pi\colon \mathfrak{C}\to \pic, \; s_i\colon \pic\to \mathfrak{C}, \; \mathfrak{L}\to \mathfrak{C}, \; \omega_\pi \to \mathfrak{C}
\end{equation} 
be the universal curve, the $i$-th section, the universal line bundle and the relative dualizing sheaf of $\pi$. The following operational Chow classes on $\pic$ are obtained from the universal structure \eqref{eq:tautStrpic}:
\begin{itemize}
    \item $\psi_i = c_1(s_i^* \omega_\pi)\in A^1_{\mathsf{op}}(\pic)\, $,
    \item $\xi_i = c_1(s_i^* \mathfrak{L})\in A^1_{\mathsf{op}}(\pic)
    \, $,
    \item $\eta= \pi_*\left(c_1(\mathfrak{L})^2\right)
    \in A^{1}_{\mathsf{op}}(\pic)
    \, $.
\end{itemize}

Let $A = (a_1,\ldots,a_n) \in \mathbb{Z}^n$ be a vector of integers satisfying 
\begin{equation}\label{eq:drAlinear}
    \sum_{i} a_i =d\,,
\end{equation}
where $d$ is the degree of the line bundle. We denote by $P_{g,A,d}^{c,r}$ the codimension $c$ component of the class 
\begin{multline*}
\hspace{-10pt}\sum_{
\substack{\Gamma\in \mathsf{G}_{g,n,d} \\
w\in \mathsf{W}_{\Gamma,r}}
}
\frac{r^{-h^1(\Gamma_\delta)}}{|\Aut(\Gamma_\delta)|}
\;
j_{\Gamma*}\Bigg[
\prod_{i=1}^n \exp\left(\frac12 a_i^2 \psi_i + a_i \xi_i \right)
\prod_{v \in V(\Gamma_\delta)} \exp\left(-\frac12 \eta(v) \right)
\\ \hspace{+10pt}
\prod_{e=(h,h')\in E(\Gamma)}
\frac{1-\exp\left(-\frac{w(h)w(h')}2(\psi_h+\psi_{h'})\right)}{\psi_h + \psi_{h'}} \Bigg]\, .
\end{multline*} 
We refer to~\cite{BHPSS20} for details about the notations. This expression is polynomial in $r$ when $r$ is sufficiently large. Let $P^c_{g,A,d}$ be the constant part of $P^{c,r}_{g,A,d}$. 
\begin{theorem} (\cite[Theorem 8]{BHPSS20})\label{thm:picDRrelation}
$P^c_{g,A,d} = 0$ for all $c>g$ in $A^c_{\mathsf{op}}(\pic)$.
\end{theorem}
After restricting  $P^c_{g,A,d}$ to (\ref{eq:drAlinear}), this expression is a polynomial in $a_1,\ldots,a_{n-1}$. The polynomiality will be used to get refined relations.

Let $L$ be a line bundle on $S$ with degree \[\int_\beta c_1(L)=d\,.\] The choice of a line bundle $L$ induces a morphism
\[\varphi_L \colon \mgn(S,\beta) \to \pic, \,\, [f\colon C \to S] \mapsto (C,f^*L)\,.\]
Then Theorem~\ref{thm:picDRrelation} gives relations
\begin{equation}\label{eq:DR}
    P^c_{g,A,d}(L)=\varphi_{L}^*P^c_{g,A,d}\cap [\mgn(S,\beta)]^{red} =0 \text{ for all } c>g
\end{equation}
in $A_{g+n-c}\big(\mgn(S,\beta)\big)$. 
\subsection{Compatibility II} \label{sec:compatibilityII}
The relations among descendent potentials coming from tautological relations on $\mgn(S,\beta)$ are compatible with the multiple cover formula. This follows from two observations. Firstly, the splitting behavior of the reduced class, discussed in Section~\ref{sec:compatibilityI}, is crucial. It is already crucial to justify compatibility with respect to boundary restriction for tautological classes pulled back from $\mgn$. For tautological relations on $\mgn(S,\beta)$, a second fact, which we want to explain below, is essential for the compatibility. 

For $c>g>0$, $A\in \BZ^n$ and $b\in \BZ$, consider the series of relations \[P^c_{g,bA,db}(L^{\otimes b})=0\] obtained by tensoring the line bundle $L$ by $b$ times. For each coefficient of a monomial in $a_i$-variables, this expression is polynomial in $b$ and hence each of $b$-variable is a relation.
As a consequence, each term of a relation $P^c_{g,A,m}(F)$ gives the same value of
\[m^{\deg - \degm} \,, \]
where $\deg(\xi)= 1$ and $\degm(\xi) = 0$, as in Definition~\ref{def:degm}. The same holds true with the roles of $F$ and $W$ interchanged. Thus, the relations are compatible with the operator
\[ m^{\deg - \degm} \bT_{m,2g-2+\degm}\,, \]
which gives the multiple cover formula in Conjecture~\ref{conj:mcf}.
\subsection{Initial condition}\label{sec:RevisitKKV}
The Katz--Klemm--Vafa (KKV) formula implies that the generating series of $\lambda_g$-integrals $$\bF_{g,m}\big(\lambda_g; \emptyset \big)$$ satisfy the multiple cover formula \cite{PT16}. Here, $\lambda_g = c_g(\BE_g)$ is the top Chern class of the rank $g$ Hodge bundle $\BE_g$ on $\Mbar_g(S,\beta)$. The KKV formula will be the starting point of our genus induction. 

The class $\lambda_g$ is a tautological class by the Grothendieck--Riemann--Roch computation (\cite{FP05}) but the formula is rather complicated. Instead we use an alternative expression of $\lambda_g$ in terms of double ramification cycle, proven in \cite{JPPZ17}. We recall that the class $(-1)^g\lambda_g$ is equal to the double ramification cycle $\dr_g(\emptyset)$ with the empty condition. By \cite[Theorem 1]{JPPZ17} the class $\dr_g(\emptyset)$ can be written as a graph sum of tautological classes {\em without} $\kappa$-classes. 
\begin{proposition}\label{prop:initial}The multiple cover formula holds in genus~$0$ and genus~$1$ for all $m\geq 1$.
\end{proposition} 
\begin{proof}
When $g=0,1$, the tautological ring $R^*(\Mbar_{g,n})$ is additively generated by boundary strata (\cite{Keel92,Petersen14}). Thus, one can replace descendents $\alpha\in R^*(\Mbar_{g,n})$ by classes in $H^*(S)$. By the divisor equation and the dimension constraint, we can reduce to the case $\bF_{0,m}\big(\emptyset\big)$ and $\bF_{1,m}(\tau_0(\pt))$. The genus~$0$ case is covered by the full Yau--Zaslow formula~\cite{KMPS10,PT16}. The genus~$1$ case follows from the genus~$2$ KKV formula. Using the boundary expression of $\lambda_2$ on $\Mbar_2$, we have
   \begin{align*}
        \bF_{2,m}\big(\lambda_2;\emptyset \big)&= \frac{1}{240}\bF_{1,m}\big(\psi_1;\Delta_S\big)+\frac{1}{1152}\bF_{0,m}\big(\,;\Delta_S,\Delta_S\big)\\
        &=\frac{1}{10}\bF_{1,m}\big(\tau_0(\pt)\big)+\frac{1}{60}\Dq^2\bF_{0,m}\big(\emptyset\big)\,,
    \end{align*}
    where $\Delta_S\subset S\times S$ is the diagonal class. Therefore, $\bF_{1,m}\big(\tau_0(\pt)\big)$ satisfies Conjecture~\ref{conj:mcf}.
\end{proof}
In the argument below, we will use tautological relations on $\Mbar_{g,n}$ which are recently obtained by $r$-spin relations. For convenience, we summarize the result.
\begin{proposition}(\cite{KLLS18}) Let $g\geq 2$ and $n\geq 1$. Consider tautological classes on $\Mbar_{g,n}$.
\begin{enumerate}[itemsep=6pt, label=(\roman*)]\label{prop:trr}
    \item  (Topological Recursion Relations) Any monomial of $\psi$-classes of degree at least $g$ can be represented by a  tautological class supported on boundary strata without $\kappa$-classes.
    \item Any tautological class of degree $g-1$ can represented by a sum of a linear combination of $\psi_1^{g-1}, \ldots, \psi_n^{g-1}$ and a tautological class supported on boundary strata.
\end{enumerate} 
\end{proposition}
\begin{proof}
The proof of (i) follows from the proof of \cite[Lemma 5.2]{KLLS18} (see also \cite[page 3]{CJWZ}). By \cite[Proposition 3.1]{KLLS18} (or  \cite[Theorem 1.1]{BSZ16}) the degree $g-1$ part $R^{g-1}(\mathcal{M}_{g,n})$ is spanned by $\psi_1^{g-1}, \ldots, \psi_n^{g-1}$. Since relations used in the proof are all tautological, the boundary expression is tautological and thus we obtain (ii).
\end{proof}
Together with the boundary expression for $\lambda_{g+1}$ we obtain the following more general consequence of the KKV formula:
\begin{proposition}\label{prop:initialfromKKV}
Let $m\geq 1$ and $g\geq 1$. Assume the multiple cover formula Conjecture~\ref{conj:mcf} holds for $m$ and all descendents of genus~$< g$. Then Conjecture~\ref{conj:mcf} holds for 
\[ \bF_{g,m}\big( \tau_{g-1}(\pt)\big) \,. \]
\end{proposition}
\begin{proof} Let $\delta\in R^1(\Mbar_{g})$ be the boundary divisor corresponding to a curve with nonseparating node. Denote two half edges as $h$ and $h'$. Recall that $(-1)^g\lambda_g$ is equal to the double ramification cycle $\dr_g(\emptyset)$ with the empty condition. We use this formula for genus $g+1$. By \cite[Theorem 1]{JPPZ17},
\begin{align*} &(-1)^{g+1}\lambda_{g+1} = \dr_{g+1}(\emptyset) \\
&= \frac{1}{2}\Big[-\frac{1}{(g+1)!}\sum_{w=0}^{r-1}\big(\frac{w^2}{2}(\psi_h+\psi_{h'})\big)^g\Big]_{r^1}\delta+ \textup{ lower genus}\,,\end{align*}
where $[\cdots]_{r^1}$ is the coefficient of the linear part of a polynomial in $r$. The leading term is nonzero by Faulhaber's formula. 

By Proposition \ref{prop:trr} (i) any $\psi$-monomial in $R^{\geq g}(\Mbar_{g,n})$ can be represented by a sum of tautological classes supported on boundary strata without $\kappa$ classes. There is only one graph with a genus~$g$ vertex (with a rational component carrying both markings). The graph is decorated with a polynomial of degree~$g-1$ in $\psi$- and $\kappa$-classes. By Proposition \ref{prop:trr} (ii) this tautological class can be represented by a sum of a multiple of $\psi^{g-1}$ and tautological classes supported on boundary strata.
We find that \footnote{The number under each vertex is the genus and legs correspond to markings.}
\begin{equation*}
    (\psi_1+\psi_2)^g = c\,\begin{tikzpicture}[scale=0.7, baseline=-3pt,label distance=0.3cm,thick,
    virtnode/.style={circle,draw,fill=black,scale=0.05}, 
    nonvirt node/.style={circle,draw,fill=black,scale=0.5} ]
    \node [nonvirt node,label=below:$g$] (A) {};
    \node at (2,0) [virtnode, label=below:$0$] (B) {};
    \draw [-] (A) to (B);
    \node at (2.7,.5) (m1) {};
    \node at (2.7,-.5) (m2) {};
    \node at (3,.5) (n1) {$1$};
    \node at (3,-.5) (n2) {$2$};
    \node at (.6,.4) (n3) {$\psi^{g-1}$};
    \draw [-] (B) to (m1);
    \draw [-] (B) to (m2);    
    \end{tikzpicture}
    + \textup{ lower genus}
\end{equation*}
in $R^g(\Mbar_{g,2})$ for some $c\in \BQ$. Therefore, it suffices to prove that $c$ is nonzero. Recall that $\lambda_g\lambda_{g-1}$ vanishes on $\Mbar_{g,n}\setminus M^{rt}_{g,n}$, so
\begin{equation*}
    \int_{\Mbar_{g,2}}(\psi_1+\psi_2)^g\lambda_g\lambda_{g-1} = c\,\int_{\Mbar_{g,1}}\psi_1^{g-1}\lambda_g\lambda_{g-1}\,.
\end{equation*}
The left hand side of the equation is nonzero by \cite[Lemma 8]{JPPZ17}, which concludes the proof.
\end{proof}

We now consider the case of genus two. By the Getzler--Ionel vanishing on $\Mbar_{2,n}$, the dimension constraint, and the divisor equation any descendent insertion reduces to the following three cases:
\[\bF_{2,m}\big (\tau_1(\pt)\big )\,, \quad \bF_{2,m}\big( \tau_0(\pt)^2 \big)\,, \quad \bF_{2,m}\big(\tau_1(\gamma)\tau_0(\pt)\big) \quad  \text{with }\gamma\in H^2(S)\,. \]
The first case is treated in Proposition~\ref{prop:initialfromKKV} and follows from the KKV formula in genus three and lower genus. The second case for $m=2$ is treated as part of the proof of Theorem~\ref{thm:qmod} in Section~\ref{sec:ProofQmodHAE}. We use the double ramification relation~\eqref{eq:DR} to prove the multiple cover formula for the third case. The point class $\pt$ will be obtained as the product of $F$ and $W$.
\begin{proposition}\label{prop:genustwo}
For $\gamma\in H^2(S)$, the generating series $\bF_{2,m}(\tau_1(\gamma)\tau_0(\pt))$ satisfies Conjecture~\ref{conj:mcf}.
\end{proposition}
\begin{proof} 
We will use relations in $A_{2+n-3}\left(\Mbar_{2,n}(S,\beta)\right)$:
\[P^3_{2,A,m}(F) = 0 \]
associated to the line bundle $\CO_S(F)$ on~$S$. More precisely, we will distinguish two cases $\gamma\in U$ and\ $\gamma\in U^\perp$ and set respectively
\[ A= (a_1\,, m-a_1)\,,\quad A= (a_1\,,a_2\,, m-a_1-a_2)\,. \]
Refined relations are then obtained by considering particular monomials in the $a_i$, as outlined in the previous section. The $\eta$-class vanishes in this case because $\langle F,F\rangle = 0$ and, for the same reason, $\xi_i^2$ vanishes. Define
\[X = \bF_{2,m}\big(\tau_1(\gamma)\tau_0(\pt)\big)\,.\]
The case $\gamma=F$ is treated first. As explained in Section~\ref{sec:DRrelation}, the tautological relations are polynomial in $a_i$ and we may obtain a refined relation by considering the $[a_1^4]$-coefficient of 
\[\vspace{12pt}P^3_{2,A,m}(F)|_{a_2=m-a_1}\,.\] 

We will only need to consider boundary strata which both:
\begin{itemize}
    \item contribute to $X$ and
    \item contribute to the $[a_1^4]$-coefficient.
\end{itemize}

These two properties simplify the calculation significantly. By the splitting property of the reduced class, a relevant boundary stratum is a tree with one genus~$2$ vertex and contracted genus~$0$ components. The integrals are given by the intersection product of the corresponding insertions. In the case with only two markings, the only relevant stratum is\footnote{The genus~$2$ vertex is represented by a filled node and other nodes represent genus~$0$ vertices. Labeled half-edges correspond to markings.}
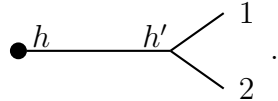
\begin{equation*}
    \begin{tikzpicture}[baseline=-3pt,label distance=0.5cm,thick,
    big/.style={circle,fill=black,draw,scale=0.5}, 
    small/.style={circle,draw,fill=black,scale=0.3} ]
    \node [big] (A) {};
    \draw [-] (A) to (2,0);
    \node at (.3,.2) {$h$};
    \node at (1.8,.2) {$h'$};
    \node at (3,.5) {$1$};
    \node at (3,-.5) {$2$};
    \draw (2,0) -- (2.7,.5);
    \draw (2,0) -- (2.7,-.5);
    \end{tikzpicture}\,.
\end{equation*}
The weight factor for this stratum is
\[\frac{w(h)w(h')}{2} = -\frac{m^2}{2}\,. \]
This stratum, therefore, cannot contribute to the $[a_1^4]$-coefficient, since $\psi$-classes on the genus~$0$ component vanish. It remains to determine the contributions from the trivial graph
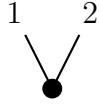
\begin{equation*}
    \begin{tikzpicture}[baseline=-3pt,label distance=0.5cm,thick,
     big/.style={circle,fill=black,draw,scale=0.6}, 
    virtnode/.style={circle,draw,scale=0.5}, 
    nonvirt node/.style={circle,draw,fill=black,scale=0.5} ]
    \node at (0,-.5)[big] (A) {};
    \node at (-.5,.5) (m1) {$1$};
    \node at (.5,.5) (m2) {$2$};
    \draw [-] (A) to (m1);
    \draw [-] (A) to (m2);    
    \end{tikzpicture}
\end{equation*}
We will order the terms by the total degree $\deg(\psi)$ in the $\psi$-classes.

\begin{enumerate}\setcounter{enumi}{-1}
    \item $\deg(\psi) = 0$. The relation we consider is of codimension three. This case is therefore impossible by virtue of $\xi_i^2=0$.
    \item $\deg(\psi) = 1$. This case results in non-trivial terms, discussed below.
    \item $\deg(\psi) \ge 2$. We may apply Proposition \ref{prop:trr} (i) to reduce to the descendent $\bF_{2,m}\big(\tau_1(\pt)\big)$. This descendent is covered by Proposition~\ref{prop:initialfromKKV}.

\end{enumerate}

Therefore, up to lower genus data, the $[a_1^4]$-coefficient is \[-\frac{1}{2}\psi_1\xi_1\xi_2-\frac{1}{2}\psi_2\xi_1\xi_2\,.\]
Integrating \[\ev^*_2(W) P^3_{2,A,m}(F)|_{a_2=m-a_1}\] against the reduced class, we find (up to lower genus data)
\[- \frac{1}{2} X-\frac{m}{2}\bF_{2,m}\big(\tau_1(\pt) \big)\,,\]
where the second term is obtained by application of the divisor equation. We thus find that $X$ is a linear combination of terms which satisfy Conjecture~\ref{conj:mcf}. Switching the role of $F$ and $W$, we obtain the same result for~$\gamma=W$.

Next, we consider $\gamma\in U^\perp$. The following vanishing of intersection products will be used frequently:
\[ \langle \gamma, F \rangle = 0\,, \quad \langle \gamma, W \rangle = 0\,,\quad  \langle \gamma, \beta \rangle=0\,.\]
 We use a similar argument as above, this time, however, we use three markings and consider the $[a_1^3 a_2]$-coefficient of\footnote{We are grateful to the referee for pointing out a mistake in an earlier version of the text. It has become clear that the choice of monomial, leading to non-trivial relations, is a very subtle one. Symmetry in the $a_i$ and the insertions causes cancellation in many cases. We plan to come back to this in future work.}
\begin{equation}\label{eq:genustwovanishing}\vspace{12pt}
    \ev_1^*(\gamma)\ev_2^*(W)P_{2,A,m}^3(F)|_{a_3=m-a_1-a_2}\,.
\end{equation}
By the above vanishing of intersection products, the only possible trees with non-trivial contribution are 
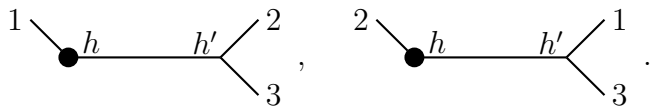
\begin{equation*}
    \begin{tikzpicture}[baseline=-3pt,label distance=0.5cm,thick,
    big/.style={circle,fill=black,draw,scale=0.6}, 
    virtnode/.style={circle,draw,scale=0.5}, 
    nonvirt node/.style={circle,draw,fill=black,scale=0.5} ]
    \node [big] (A) {};
    \draw [-] (A) to (2,0);
    \node at (2.7,.5) (m1) {$2$};
    \node at (2.7,-.5) (m2) {$3$};
    \node at (-.7,.5) (m3) {$1$};
    \node at (.3,.2) {$h$};
    \node at (1.8,.2) {$h'$};
    \draw (2,0) -- (2.5,.5);
    \draw (2,0) -- (2.5,-.5);
    \draw (A) -- (-.5,.5);
    \end{tikzpicture}\,,\quad 
    \begin{tikzpicture}[baseline=-3pt,label distance=0.5cm,thick,
     big/.style={circle,fill=black,draw,scale=0.6}, 
    virtnode/.style={circle,draw,scale=0.5}, 
    nonvirt node/.style={circle,draw,fill=black,scale=0.5} ]
    \node [big] (A) {};
    \draw [-] (A) to (2,0);
    \node at (2.7,.5) (m1) {$1$};
    \node at (2.7,-.5) (m2) {$3$};
    \node at (-.7,.5) (m3) {$2$};
    \node at (.3,.2) {$h$};
    \node at (1.8,.2) {$h'$};
    \draw (2,0) -- (2.5,.5);
    \draw (2,0) -- (2.5,-.5);
    \draw (A) -- (-.5,.5);
    \end{tikzpicture}\,.
\end{equation*}

The weight factor for the right stratum is
\[\frac{w(h)w(h')}{2} = -\frac{(m-a_2)^2}{2}\,. \]
Since $\psi$-classes on the genus~$0$ component vanish, the power of $a_1$ in any monomial obtained from this stratum is bounded by two. The contribution to the $[a_1^3 a_2]$-coefficient is, therefore, zero.

Next, we explain the contributions from the left stratum. Note that the left vertex is of genus~$2$ with two markings and we may apply the same reasoning as in the discussion for $\gamma=F$ above. Here, the $\deg(\psi)=0$ term $\xi_1\xi_2$ has trivial contribution due to $\langle \gamma,F\rangle = 0$. The $\deg(\psi)=1$ terms $\psi_h \xi_2$, $\psi_h\xi_3$ have vanishing contribution by application of the divisor equation for $\gamma$. Non-trivial contributions are obtained only from 
\[ \psi_1 \xi_2\,,\quad \psi_1 \xi_3\,. \]
These two terms have contributions 
\[-\frac{(m-a_1)^2}{4}a_1^2a_2 X\,,\quad -\frac{(m-a_1)^2}{4}a_1^2a_3X\,.\]
The $[a_1^3 a_2]$-coefficients, however, cancel due to $a_3 = m-a_1-a_2$. It remains to determine the contributions from the trivial graph:
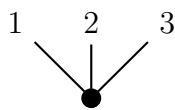
\begin{equation*}
    \begin{tikzpicture}[baseline=-3pt,label distance=0.5cm,thick,
    big/.style={circle,fill=black,draw,scale=0.6}, 
    virtnode/.style={circle,draw,scale=0.5}, 
    nonvirt node/.style={circle,draw,fill=black,scale=0.5},]
    \node at (0,-.5)[big] (A) {};
    \node at (-1,.5) (m1) {$1$};
    \node at (0,.5) (m2) {$2$};
    \node at (1,.5) (m3) {$3$};
    \draw [-] (A) to (m1);
    \draw [-] (A) to (m2);   
    \draw [-] (A) to (m3);    

    \end{tikzpicture}
\end{equation*}
As above, we order the terms by the total degree $\deg(\psi)$ in the $\psi$-classes.

\begin{enumerate}\setcounter{enumi}{-1}
    \item $\deg(\psi) = 0$. The relation we consider is of codimension three. Since $\xi_i^2=0$, the class $\xi_1$ must appear. This term, however, vanishes due to $\langle \gamma, F\rangle=0$.
    \item $\deg(\psi) = 1$. This case results in non-trivial terms corresponding to $\psi_1$ or $\psi_3$, discussed below. The choice of the monomial $[a_1^3a_2]$ excludes the appearance of $\psi_2$.
    \item $\deg(\psi) = 2$. This case results in non-trivial terms corresponding to $\psi_1\psi_3$ or $\psi_3^2$, discussed below. The choice of the monomial $[a_1^3a_2]$ excludes the appearance of $\psi_1^2$.
    \item $\deg(\psi) = 3$. As above, this case reduces to the descendent $\bF_{2,m}\big(\tau_1(\pt)\big)$ which is covered already.
\end{enumerate}

The contributions from $\deg(\psi)\in\{1,2\}$ are:
\begin{align*}
    \psi_1\xi_2\xi_3 \quad\to\quad  &\frac{1}{2}a_1^2 a_2 a_3 \bF_{2,m}\big(\tau_1(\gamma)\tau_0(\pt)\tau_0(F)\big)\\
    =\,&\frac{1}{2} a_1^2a_2 (m-a_1-a_2) m X\,, \\
    \psi_3\xi_2\xi_3\quad\to\quad  &\frac{1}{2}a_3^2 a_2 a_3 \bF_{2,m}\big(\tau_0(\gamma)\tau_0(\pt)\tau_1(F)\big)\\
    =\,&0\,, \\
    \psi_1\psi_3\xi_2 \quad\to\quad  &\frac{1}{2}a_1^2 \frac{1}{2}a_3^2a_2 \bF_{2,m}\big(\tau_1(\gamma)\tau_0(\pt)\tau_1(\one)\big)\\
    =\,&a_1^2a_2 (m-a_1-a_2)^2 X\,, \\
    \psi_1\psi_3\xi_3 \quad\to\quad  &\frac{1}{2}a_1^2\frac{1}{2}a_3^3\bF_{2,m}\big(\tau_1(\gamma)\tau_0(W)\tau_1(F)\big) \\
    =\,&\frac{1}{4}a_1^2(m-a_1-a_2)^3 X + \text{ (lower genus)}\,, \\
    \psi_3^2\xi_2 \quad\to\quad & \frac{1}{8} a_3^4a_2 \bF_{2,m}\big(\tau_0(\gamma)\tau_0(\pt)\tau_2(\one)\big)\\
    =\,& \frac{1}{8} a_2 (m-a_1-a_2)^4 X\,, \\
    \psi_3^2\xi_3 \quad\to\quad &\frac{1}{8} a_3^4a_3 \bF_{2,m}\big(\tau_0(\gamma)\tau_0(W)\tau_2(F)\big)\\
    =\,&0\,.
\end{align*}
The third calculation uses the dilaton equation. All of the other calculations are obtained by application of the divisor equation. Additionally, the fourth calculation involves Proposition \ref{prop:trr}. The only stratum with a genus~$2$ vertex (i.e.\ with both markings on a contracted genus~$0$ component) has vanishing contribution due to $\langle \gamma, F\rangle = 0$ and, therefore, the relation reduces to lower genus descendents. The total contribution to $[a_1^3a_2]$ is
\[-\frac{1}{2}mX-2mX+\frac{3}{2}mX-\frac{1}{2}mX  = -\frac{3}{2}mX\,.\]
We find that $X$ is a linear combination of terms which satisfy Conjecture~\ref{conj:mcf}.

\end{proof}
\begin{remark}
 In fact, for $\gamma\in U^\perp$ the above generating series vanishes (and thus trivially satisfies the multiple cover formula). A proof in the primitive case is given in~\cite[Lemma 4]{B19}.
\end{remark}

\section{Proof of Theorem \ref{thm:qmod} and \ref{thm:hae}}\label{sec:ProofQmodHAE}

\subsection{Proof of Theorem \ref{thm:qmod}}
The proof proceeds via induction on the pair $(g,n)$ ordered by the lexicographic order: $(g',n')<(g,n)$ if 
\begin{itemize}
    \item $g'<g$ or
    \item $g'=g$ and $n'<n$\,.
\end{itemize}
Recall the dimension constraint of insertions:
\[g+n = \deg(\alpha) + \sum_i \deg (\gamma_i)\,.\]
We separate the proof into several steps.

\noindent\textbf{Case 0.} The genus~$0$ case is covered by Proposition~\ref{prop:initial}. This serves as the start for our induction.

\noindent\textbf{Case 1.} If all cohomology classes $\gamma_i$ satisfy $\deg(\gamma_i)\leq 1$, then $\deg(\alpha)\geq g$ and by the strong form of Getzler--Ionel vanishing~\cite[Proposition 2]{FP05} we have $\alpha = \iota_* \alpha'$ with $\alpha'\in R^*(\partial\mgn)$ and $\iota\colon \partial\Mbar_{g,n}\to \Mbar_{g,n}$. We are thus reduced to lower $(g,n)$.

\noindent\textbf{Case 2.} Assume $\deg(\alpha)\leq g-2$ or equivalently, there exist at least two descendents of the point class. We use the degeneration to the normal cone of a smooth elliptic fiber:
\[ S \rightsquigarrow S \cup_E (\BP^1\times E)\,.\]

We specialize the point class to the bubble $\BP^1\times E$. Let $C=C'\cup C''$ be the splitting of a domain curve appearing in the degeneration formula in Theorem~\ref{thm:degeneration}. Namely, $C'$ is the component on $S$ and $C''$ is the component on $\BP^1\times E$. We argue that this splitting has non-trivial contribution only for $g(C')<g$. If $g(C') = g$, this forces $C''$ to be a disconnected union of two rational curves. Since the degree of the curve class along the divisor is $\langle 2B+hF, F\rangle = 2$, the two descendents of the point class then force the cohomology weighted partition to be $(1,\one)^2$ on the bubble or, equivalently, $(1,\omega)^2$ for $(S,E)$. This contribution vanishes because there are no curves which can satisfy this condition (if $(1,\omega)^2$ is represented by a generic point in $E^2$, see Corollary~\ref{cor:degeneration}).\medskip

\noindent\textbf{Case 3.} Assume $\deg(\alpha) = g-1$ or equivalently, there exists only one desecendent of the point class. We may thus assume $\gamma_1 = \pt$. If $n=1, g\geq 2$, we can move $\tau_{g-1}(\pt)$ to the bubble and the genus on $S$ drops.

When $n\geq 2$, moving the point class to the bubble as in Case 2 may not reduce the genus. In particular, moving $\tau_0(\pt)$ to the bubble has non-trivial contribution from rational curves on the bubble. On the other hand, if $a\geq 1$, moving $\tau_a(\pt)$ to the bubble reduces the genus on $S$ because of the dimension constraint.

We use Buryak, Shadrin and Zvonkine's description of the top tautological group $R^{g-1}(M_{g,n})$ \cite{BSZ16}. For any $\alpha\in R^{g-1}(\Mbar_{g,n})$ the restriction of $\alpha$ to $M_{g,n}$ is a linear combination of
\begin{equation}\label{eq:ToptautGp}
    R^{g-1}(M_{g,n}) = \BQ\lan \psi_1^{g-1},\psi_2^{g-1},\ldots,\psi_n^{g-1}\ran
\end{equation}
and the boundary term is also tautological class in $R^{g-1}(\partial \Mbar_{g,n})$. By the divisor equation and subsequent use of (\ref{eq:ToptautGp}), we can reduce to  cases for $\leq(g,2)$. When $g\geq 3$, (\ref{eq:ToptautGp}) has a different basis
\[R^{g-1}(M_{g,2}) = \BQ\lan \psi_1^{g-1},\psi_1\psi_2^{g-2}\ran\]
which is an easy consequence of the generalized top
intersection formula. Therefore, we may assume the descendent of the point class is of the form $\tau_{a}(\pt)$ with $a\geq 1$. Now, specializing this insertion to the bubble $\BP^1\times E$ reduces the genus and hence the same argument as in Case 2 applies. The genus~$2$ case is covered by Proposition~\ref{prop:genustwo}.\medskip

\noindent\textbf{Relative vs.\ absolute.} We reduced to invariants for $(S,E)$ with genus $g'< g$. As explained in the proof of~\cite[Lemma 31]{MPT10} (see also ~\cite{MP06}), the degeneration formula provides an upper triangular relation between absolute and relative invariants for all pairs $\leq (g',n')$. Thus, our induction applies.

\subsection{Proof of Theorem \ref{thm:hae}}
We argue by showing that each induction step in the proof of Theorem~\ref{thm:qmod} is compatible with the holomorphic anomaly equation. Nontrivial step appears when the degeneration formula is used. From the compatibility result Proposition~\ref{prop:compatibility}, we are reduced to proving the relative holomorphic anomaly equation for lower genus relative generating series $\bF^{\rel}_{g',2}$ for $(S,E)$ and relative generating series for $(\BP^1\times E,E)$. The holomorphic anomaly equation for $(\BP^1\times E,E)$ is established in~\cite{OP18}. Because of the relative vs.\ absolute correspondence~\cite{MP06}, we are reduced to proving the holomorphic anomaly equation for $\bF_{g',2}$ in genus $0, 1$ and some genus
~$2$ descendents. We proved the multiple cover formula for these cases in Section~\ref{sec:InitialCondition}, which implies the holomorphic anomaly equation by Proposition~\ref{prop:mcfimplieshae}.

\begin{remark}
Parallel argument shows that we can always reduce the proof for arbitrary descendent insertions to the case when the number of point insertions is less than or equal to $m-1$.
\end{remark}

\section{Examples}

\begin{example}
We compute $\bF_{1,2}\big(\tau_1(F)\big)$ via topological recursion in genus one and illustrate Conjecture~\ref{conj:mcf}. Let $[\delta_0]\in A^1(\Mbar_{1,1})$ be the pushforward of the fundamental class under the gluing map
\[ \Mbar_{0,3}\to \Mbar_{1,1}\,.\]
Since
\[ \psi_1 = \frac{1}{24} [\delta_0] \in A^1(\Mbar_{1,1})\,, \]
we obtain
\begin{align*} \bF_{1,1}\big(\tau_1(F)\big) &= \frac{1}{24} \bF_{0,1}\big(\tau_0(F) \tau_0(\Delta_S)\big) = \frac{1}{12} \bF_{0,1}\big(\tau_0(F) \tau_0(F\times W)\big) \\[8pt]
&= \frac{1}{12} \Dq\bF_{0,1}\,,\end{align*}
where $\Delta_S\subset S\times S$ is the diagonal class. Analogously, 
\[\bF_{1,2}\, \big(\tau_1(F)\big) = \frac{1}{24}\, \bF_{0,2}\big(\tau_0(F) \tau_0(\Delta_S)\big) =\frac{1}{3} \,\Dq\bF_{0,2}\,.\]
Using the multiple cover formula in genus zero
\[ \bF_{0,2} = \bT_2\bF_{0,1} + \frac{1023}{8192}\, \bF_{0,1}(q^2)\,,\]
we obtain
\begin{align*} \bF_{1,2}\big(\tau_1(F)\big)&=\frac{1}{3} \Dq\bF_{0,2}= 2\, \bT_{2}\frac{1}{12}\Dq \bF_{0,1} + \frac{1023}{1024}\, \bB_2 \frac{1}{12}\Dq \bF_{0,1} \\[8pt]
&= 2\, \bT_{2}\bF_{1,1}\big(\tau_1(F)\big) + (2^0-2^{-10})\, \bB_2\bF_{1,1}\big(\tau_1(F)\big)\, ,\end{align*}
in perfect agreement with Conjecture~\ref{conj:mcf} using the formula for $\bT_{2,0}$ from Lemma~\ref{lem:hecke}.
\end{example}

\begin{example}\label{ex:genus2}
We compute $\bF_{2,2}(\tau_0(\pt)^2)$ via degeneration formula and verify the multiple cover formula. The first two terms are computed by the classical geometry of K3 surface in~\cite{OP16}.
For simplicity we write $\bF_{1,2}=\bF_{1,2}(\tau_0(\pt))$. The relative invariants for $(S,E)$ can be written in terms of absolute invariants:
\begin{lemma}
\begin{enumerate}[label=(\roman*)]
    \item $\bF^{\rel}_{0,2}\big(\emptyset\mid(1,\one)^2\big) = 2 \bF_{0,2}$,
    \item $\bF^{\rel}_{1,2}\big(\emptyset\mid(1,\one),(1,\omega)\big)= \bF_{1,2}-2\bF_{0,2}\Dq C_2$,
    \item $\bF^{\rel}_{1,2}\big(\emptyset\mid(2,\one)\big)= \frac{1}{3}\Dq \bF_{0,2}-4C_2\bF_{0,2}$.
\end{enumerate}
\end{lemma}
\begin{proof}
    It is a standard computation of the relative vs.\ absolute correspondence~\cite{MP06}.
\end{proof}
The relative invariants for $(\BP^1\times E, E)$ can be computed by the Gromov--Witten invariants of $E$.
\begin{lemma}
\begin{enumerate}[label=(\roman*)]
    \item $\bG^{\rel}_{0,1}\big(\tau_0(\pt)\mid(1,\one)\big) = 1$, \, $\bG^{\rel}_{0,1}\big(\emptyset\mid(1,\omega)\big) = 1$,
    \item $\bG^{\rel}_{1,1}\big(\tau_0(\pt)\mid(1,\omega)\big)=\Dq C_2$, \, $\bG^{\rel}_{1,1}\big(\tau_0(\pt)^2\mid(1,\one)\big)=2\Dq C_2$,
    \item $\bG^{\rel}_{2,1} \big(\tau_0(\pt)^2\mid(1,\omega)\big)=(\Dq C_2)^2$,
    \item $\bG^{\rel}_{1,2}\big(\tau_0(\pt)^2\mid(2,\omega)\big) = \Dq^2 C_2$, \, $\bG^{\rel}_{1,2}\big(\tau_0(\pt)^2\mid(1,\omega)^2\big) = \Dq^3 C_2$.
\end{enumerate}
\end{lemma}
Consider the degeneration where two point insertions move to the bubble $\BP^1\times E$. By Theorem~\ref{thm:degeneration}, 
\begin{align*}
\bF_{2,2}\big(\tau_{0}(\p)^2\big) &= \big(\bF_{1,2} - 2\bF_{0,2} \Dq C_2\big)4\Dq C_2 +\big(\frac{1}{3}\Dq\bF_{0,2} - 4C_2 \bF_{0,2}\big)2\Dq^2 C_2 \\
&+ (2\bF_{0,2})\frac{1}{2}\big(\Dq^{3}C_2+4(\Dq C_{2})^2\big)\\
&=36q+8760q^2+754992q^3+36694512q^4 + \cdots \,.
\end{align*}
On the other hand, the primitive generating series \[\bF_{2,1}\big(\tau_0(\pt)^2\big)=\frac{\big(\Dq C_2\big)^2}{\Delta(q)}\] is computed in \cite{BL00} and one can apply the multiple cover formula to obtain a candidate for $\bF_{2,2}\big(\tau_{0}(\p)^2\big)$. The first few terms of the two generating series match. It is enough to conclude that the two generating series are indeed equal because the space of quasimodular forms with given weight is finite dimensional. However, it seems non-trivial to match the above formula from the degeneration with the formula provided by Conjecture~\ref{conj:mcf}.
\end{example}

\appendix 
\section{A proof of degeneration formula}\label{app:degeneration}
For a self-contained exposition, we present a proof of the degeneration formula which is parallel to the proof in \cite{MPT10,MP13}. When $m=1,2$, a proof using symplectic geometry was presented in \cite{LL05}.

\subsection*{Perfect obstruction theory}

For simplicity assume $n=0$. General cases  easily follow from this case.
Let $\epsilon\colon \CS\to \BA^1$ be the total family of the degeneration and \[\mg(\epsilon,\beta)\to \BA^1\] be the moduli space of stable maps to the expanded target $\widetilde{\CS}$. For the relative profile $\mu$, the embedding \[\iota_{\mu}\colon \mg(\CS_0,\mu)\hookrightarrow \mg(\epsilon,\beta)\] can be realized as a Cartier pseudo-divisor $(L_\mu,s_\mu)$. 

Let $\Ee_\epsilon\to \BL_{\mg(\epsilon,\beta)}$ be the perfect obstruction theory constructed in \cite{Li02}. 
Then the perfect obstruction theories $\Ee_0$ and $\Ee_\mu$ of $\mg(\CS_0,\beta)$ and $\mg(\CS_0,\mu)$ sit in exact triangles
\begin{align*}
     & L^\vee_0\to \iota_0^*\Ee_\epsilon \to \Ee_0 \xrightarrow{[1]} \\
    & L^\vee_\mu\to\iota_\mu^*\Ee_\epsilon\to \Ee_\mu\xrightarrow{[1]} \,.
\end{align*}
On $\mg(\CS_0,\mu)$, the perfect obstruction theory splits as follows. Let $\Ee_1$ and $\Ee_2$ be the perfect obstruction theory of relative stable map spaces $\mg(S/E,\beta_1)_\mu$ and $\mg(\BP^1\times E/E,\beta_2)_\mu$ respectively. There exists an exact triangle
\begin{equation}\label{eq:splitObs}
    \bigoplus_{i=1}^{l(\mu)}(N^\vee_{\Delta_E/E\times E})_i\to \Ee_1 \boxplus \Ee_2 \to \Ee_\mu \xrightarrow{[1]}
\end{equation}
where $(N^\vee_{\Delta_E/E\times E})_i$ is the pullback of the conormal bundle of the diagonal $\Delta_E\subset E\times E$ along the $i$-th relative marking.

\subsection*{Reduced class}

Let $\rho\colon\TCS\to S\times \BA^1\to S$ be the projection. By pulling back the holomorphic symplectic form on $S$ via $\rho$, one can define a cosection of the obstruction sheaf of $\Ee_\epsilon$ \[Ob_{\mg(\epsilon,\beta)}\to \CO\,,\] see \cite[Section 5]{KL11}. Dualizing the cosection gives a morphism \[\gamma\colon \CO[1]\to \Ee_\epsilon\,.\] Let $\Ee^{\red}_\epsilon$ be the cone of $\gamma$ which gives the reduced class on $\mg(\epsilon,\beta)$. Similarly we can construct \[\gamma_{\textup{rel}}\colon \CO[1]\to \Ee_1\] for the moduli space of relative stable maps $\mg(S/E,\beta)$.

\subsection*{Degeneration formula for reduced class}

Restricting $\gamma$ to $\mg(\CS_0,\beta)$ and $\mg(\CS_0,\mu)$, we get
\begin{align*}
    & \gamma_0\colon \CO[1]\to \iota^*_0 \Ee_\epsilon\to \Ee_0\\
    & \gamma_\mu\colon \CO[1]\to \iota^*_\mu \Ee_\epsilon \to \Ee_\mu
\end{align*}
where the compositions induce reduced classes. The exact triangles
\begin{align*}
    & L_0^\vee\to \iota_0^*\Ee_\epsilon^{\red}\to \Ee^{\red}_0\xrightarrow{[1]}\,,\\
    & L_\mu^\vee \to \iota_\mu \Ee_\epsilon^{\red}\to \Ee^{\red}_\mu \xrightarrow{[1]} \,,
\end{align*}
still hold.
\begin{lemma}\label{lem:potDegeneration}
We have an exact triangle 
\[N^\vee_{\Delta_{E^l}/E^l\times E^l}\to \Ee_1^{\red}\boxplus \Ee_2 \to \Ee^{\red}_\mu\xrightarrow{[1]}\]
on $\mg(\CS_0,\mu)$ compatible with the structure maps to the cotangent complex.
\end{lemma}
\begin{proof}
Consider the diagram of complexes
\begin{equation*}
\begin{tikzcd}[row sep = large, column sep = normal]
&\CO[1]\boxplus 0 \arrow[r, equal] \arrow[d, "\gamma_{\textup{rel}}\boxplus 0"] & \CO[1] \arrow[d, "\gamma_\mu" ] \\ 
\bigoplus_{i=1}^{l(\mu)}(N^\vee_{\Delta_E/E\times E})_i \arrow[r, ""] \arrow[d,equal] &\Ee_1\boxplus \Ee_2 \arrow[r, ""] \arrow[d, ""] &\Ee_\mu \arrow[d, ""]\\
\bigoplus_{i=1}^{l(\mu)}(N^\vee_{\Delta_E/E\times E})_i \arrow[r, ""] & \Ee_1^{\red}\boxplus \Ee_2\arrow[r, ""] & \Ee_\mu^{\red}
\end{tikzcd}
\end{equation*}
where the middle horizontal morphisms are the exact triangle from (\ref{eq:splitObs}). The square on the top commutes because the cosections for $\TCS$ and $(S,E)$ are both coming from the holomorphic symplectic form on $S$.  The vertical morphisms are exact triangles and hence induces a map between cones.  
\end{proof}
Now Theorem~\ref{thm:degeneration} is a direct consequence of Lemma~\ref{lem:potDegeneration}.


\bibliographystyle{myamsplain.bst}
\bibliography{refs.bib}
\end{document}